\documentclass{article}

\usepackage[T1]{fontenc}
\usepackage[utf8]{inputenc}
\usepackage{color}
\usepackage{graphicx}
\usepackage{amssymb}
\usepackage{amsmath, xfrac}
\usepackage{latexsym}
\usepackage{graphicx}
\usepackage{subcaption}
\usepackage{float}
\usepackage{array}
\usepackage{amsthm}
\usepackage{lineno}
\usepackage{enumitem}
\usepackage{tikz}
\usepackage{clrscode3e}
\usepackage{comment}
\usepackage{gensymb}
\usepackage{hyperref}
\usepackage{lineno}
\usepackage{authblk}
\usepackage{bookmark}

\edef\restoreparindent{\parindent=\the\parindent\relax}
\usepackage{parskip}
\restoreparindent

\newcounter{problem}
\newtheorem{prob}[problem]{Problem}
\newtheorem{conj}[problem]{Conjecture}

\newcounter{contSect} \numberwithin{contSect}{section}
 \numberwithin{contSub}{subsection}

\newtheorem{theorem}[contSect]{Theorem}
\newtheorem{corollary}[contSect]{Corollary}
\newtheorem{lemma}[contSect]{Lemma}

\newtheorem{claim}[contSect]{Claim}

\newtheorem{observation}[contSect]{Observation}

\usepackage[letterpaper,top=2cm,bottom=2cm,left=3cm,right=3cm,marginparwidth=1.75cm]{geometry}

\title{Covering and packing with homothets of limited capacity}
\author{Oriol Solé Pi}

\date{Facultad de Ciencias, Universidad Nacional Autónoma de México, México. oriolandreu@ciencias.unam.mx}

\begin{document}

\maketitle

\section{Abstract}

This work revolves around the two following questions: Given a convex body $C\subset\mathbb{R}^d$, a positive integer $k$ and a finite set $S\subset\mathbb{R}^d$ (or a finite Borel measure $\mu$ on $\mathbb{R}^d$), how many homothets of $C$ are required to cover $S$ if no homothet is allowed to cover more than $k$ points of $S$ (or have measure larger than $k$)? How many homothets of $C$ can be packed if each of them must cover at least $k$ points of $S$ (or have measure at least $k$)? We prove that, so long as $S$ is not too degenerate, the answer to both questions is $\Theta_d(\frac{|S|}{k})$, where the hidden constant is independent of $d$. This is optimal up to a multiplicative constant. Analogous results hold in the case of measures. Then we introduce a generalization of the standard covering and packing densities of a convex body $C$ to Borel measure spaces in $\mathbb{R}^d$ and, using the aforementioned bounds, we show that they are bounded from above and below, respectively, by functions of $d$. As an intermediate result, we give a simple proof the existence of weak $\epsilon$-nets of size $O(\frac{1}{\epsilon})$ for the range space induced by all homothets of $C$. Following some recent work in discrete geometry, we investigate the case $d=k=2$ in greater detail. We also provide polynomial time algorithms for constructing a packing/covering exhibiting the $\Theta_d(\frac{|S|}{k})$ bound mentioned above in the case that $C$ is an Euclidean ball. Finally, it is shown that if $C$ is a square then it is NP-hard to decide whether $S$ can be covered using $\frac{|S|}{4}$ squares containing $4$ points each.

\newpage
\tableofcontents

\newpage
\section{Introduction}\label{chap:intro}

This work revolves around the two following natural questions: Given a convex body $C$, a finite set of points $S\subset\mathbb{R}^d$, and a positive integer $k$, how many homothets of $C$ are required in order to cover $S$ if each homothet is allowed to cover at most $k$ points? (covering question). How many homothets can be packed if each of them must cover at least $k$ points? (packing question). We shall denote these two quantities by $f(C,k,S)$ and $g(C,k,S)$, respectively. Analogous functions can be defined if, instead of $S$, we consider a finite Borel measure $\mu$ on $\mathbb{R}^d$. As far as we know, these questions have not been studied before in such generality.

Clearly, $f(C, k, S)\geq \frac{|S|}k$ and $g(C, k, S)\leq \frac{|S|}k$, and it is easy to construct, for any $C$ and $k$, arbitrarily large sets $S$ for which equality holds (take, for example, any set formed by some clusters which lie far away from each other and contain $k$ points each). Perhaps surprisingly, under some mild assumptions on $S$ (or $\mu$) $f$ and $g$ will also be bounded from above and below, respectively, by linear functions of $\frac{|S|}{k}$ (or $\frac{\mu(\mathbb{R}^d)}{k})$, that is, $f(C,k,S)=O_d(\frac{|S|}k)$ and $f(C,k,S)=\Omega_d(\frac{|S|}k)$, where the hidden constant depends only on $d$. For Euclidean balls, both of these bounds follow from the Besicovitch covering theorem, first shown by Besicovitch~\cite{besicovitch_1945} in the planar case and later extended to higher dimensions and more general objects by Morse~\cite{Morse_1947} and Bliedtner and Loeb~\cite{BliedtnerLoeb}, this is discussed in further detail in the following section. We give a proof of the desired bounds for $f$ and $g$ that does not rely on the Besicovitch covering theorem.

The standard packing and covering densities depend implicitly on the Lebesgue measure. We introduce a generalization of covering and packing densities to Borel measure spaces in $\mathbb{R}^d$. Then, using the aforementioned bounds on $f$ and $g$, we show that for every $C$ and every nice enough measure, these covering and packing densities are bounded from above and below, respectively, by two constants that depend only on $d$. When restricted to the Lebesgue measure, this is equivalent to the relatively simple fact, mentioned earlier, that the standard covering and packing densities are accordingly bounded by a function of $d$. 

For squares, disks and triangles in the plane, the case $k=2$ has received some attention in discrete geometry (\cite{delaunaytoughness,matchingsquares,matchingnp,trianglematchingsfirst,trianglematchings,strongmatchings}). Continuing this trend, we separately study the case $d=k=2$ for more general convex bodies.

We discuss algorithms for efficiently packing and covering with homothets that contain at least $k$ and at most $k$ points, respectively. Bereg et al. \cite{matchingnp} showed that, even for $k=2$, finding an optimal packing with such homothets of a square is NP-hard, we complement this result by showing that the covering problem is also NP-hard in the case of squares \cite{matchingnp}.

At some point in this work, we require some basic tools from the study of Delaunay triangulations and $\epsilon$-nets.

\section{Preliminaries}\label{chap:prelim}

\subsection{Basic notation and definitions}\label{sec:notation}
A set $C\subset\mathbb{R}^d$ is a \textit{convex body} if it is convex, compact and its interior is nonempty. Furthermore, if the boundary of a convex body contains no segment of positive length, then we say that it is a \textit{strictly convex body}. Given any set $C\subset\mathbb{R}^d$, an \textit{homothetic copy} of $C$ (or, briefly, an \textit{homothet} of $C$) is any set of the form $\lambda C+x=\lbrace \lambda c+x : c\in C\rbrace$ for some $x\in\mathbb{R}^d$ and $\lambda>0$\footnote{Some texts ask only that $\lambda\neq 0$. We consider only positive homothets.}; the number $\lambda$ is said to be the \textit{coefficient of the homothety}\footnote{An homothety maps every point $p\subset \mathbb{R}^d$ to $\lambda p+x$, for some $x\in\mathbb{R}^d$, $\lambda\neq 0$.}. From here on, $C$ will stand for a convex body in $\mathbb{R}^d$.

We say that a set of points $S\subset\mathbb{R}^d$ is \textit{non-$t/C$-degenerate} if the boundary of any homothet of $C$ contains at most $t$ elements of $S$. We say that $S$ is in $C$-\textit{general position} if it is non-$(d+2)/C$-degenerate.

All measures we consider in this work are Borel measures in $\mathbb{R}^d$ which take finite values on all compact sets. A measure $\mu$ is \textit{finite} if $\mu(\mathbb{R}^d)<\infty$. We say that a measure is \textit{non-$C$-degenerate} if it vanishes on the boundary of every homothet of $C$. Notice that, in particular, any absolutely continuous measure (with respect to the Lebesgue measure) is non-$C$-denegerate. Finally, a measure $\mu$ is said to be \textit{$C$-nice} if it is finite, non-$C$-degenerate, and there is a ball $K\subset\mathbb{R}^d$ such that $\mu(K)=\mu(\mathbb{R}^d)$. 

Given a set of points $S\subset\mathbb{R}^d$ (resp. a measure $\mu$) and a positive number $k$, an homothet will be called a $k^+/S$-\textit{homothet} ($k^+/\mu$-\textit{homothet}) if it contains at least $k$ elements of $S$ (if $\mu(C')\geq k $). Similarly, $k^-/S$-\textit{homothets} and $k^-/\mu$-\textit{homothets} are homothets that contain at most $k$ points and have measure at most $k$, respectively.

For any finite set $S$ and any positive integer $k$, define $f(C,k,S)$ as the least number of $k^-/S$-homothets of $C$ that can be used to cover $S$, and $g(C,k,S)$ as the maximum number of interior disjoint $k^+/S$-homothets of $C$ that can be arranged in $\mathbb{R}^d$. Similarly, for any $C$-nice measure $\mu$ and any real number $k>0$, define $f(C,k,\mu)$ as the the minimum number number of $k^-/\mu$-homothets that cover $K$, where $K$ denotes the ball such that $\mu(B)=\mu(\mathbb{R}^d)$\footnote{Strictly speaking, $f$ is a function of $C,k,\mu$ and $K$. This will not cause any trouble, however, since all the properties that we derive for $f$ will hold independently of the choice of $K$.}, and define $g(C,k,\mu)$ as the maximum number of interior disjoint $k^+/\mu$-homothets that can be arranged in $\mathbb{R}^d$. It is not hard to see that, since  $S$ is finite and $\mu$ is $C$-nice, $f$ and $g$ are well defined and take only non-negative integer values.

Next, we introduce $\alpha$-\textit{fat} convex objects. For any point $x\in\mathbb{R}^d$ and any positive $r$, let $B(x,r)$ denote the open ball with center $x$ and radius $r$ (with the Euclidean metric). We write $B^d$ for $B(O,1)$, where $O$ denotes the origin (this way, $rB^d$ denotes the ball of radius $r$ centered at the origin). Given $\alpha\in (0,1]$, a convex body $C$ will be said to be $\alpha$-\textit{fat} if $B(x,\alpha r)\subseteq C\subseteq B(x,r)$ for some $x$ and $r$. The following well-known fact (e.g.~\cite{fatlemma,fattening}) will play a key role in ensuring that the hidden constants in the bounds of $f$ and $g$ are independent of $C$.

\begin{lemma}\label{teo:fat}
Given a convex body $C\subset\mathbb{R}^d$, there exists a non-singular affine transformation $T$ such that $T(C)$ is $1/d$-fat. More precisely, $B^d\subseteq T(C)\subseteq dB^d$.
\end{lemma}

By a \textit{planar embedded graph} we mean a planar graph drawn in the plane so that the vertices correspond to points, the edges are represented by line segments, no edge contains a vertex other than its endpoints, and no two edges intersect, except possibly at a common endpoint.

As usual, $\mathbb{S}^{d-1}$ stands for the unit sphere in $\mathbb{R}^d$ centered at the origin. We denote the Euclidean norm of a point $x\in\mathbb{R}^d$ by $|x|$. Throughout this text we use the standard $O$ and $\Omega$ notations for asymptotic upper and lower bounds, respectively. The precise definitions can be found, for example, in any introductory textbook on algorithm design and analysis.

\subsection{Packing and covering densities}\label{sec:packcover}
A family of sets in $\mathbb{R}^d$ forms a \textit{packing} if their interiors are disjoint, and it forms a \textit{covering} if their union is the entire space. The \textit{volume} of a measurable set $A\subset\mathbb{R}^d$ is simply its Lebesgue measure, which we denote by $\text{Vol}(A)$. The precise definitions of packing and covering densities vary slightly from text to text; for reasons that will become apparent later, we follow~\cite{packandcover}.

Let $\mathcal{A}$ be a family of sets, each having finite volume, and $D$ a set with finite volume, all of them in $\mathbb{R}^d$. The \textit{inner density} $d_{\text{inn}}(\mathcal{A}|D)$ and \textit{outer density} $d_{\text{out}}(\mathcal{A}|D)$ are given by $$d_{\text{inn}}(\mathcal{A}|D)=\frac{1}{\text{Vol}(D)}\sum_{A\in\mathcal{A},A\subset D}\text{Vol}(A),$$ $$d_{\text{out}}(\mathcal{A}|D)=\frac{1}{\text{Vol}(D)}\sum_{A\in\mathcal{A},A\cap D\neq\emptyset}\text{Vol}(A).$$ 

We remark that these densities may be infinite.

The \textit{lower density} and \textit{upper density} of $\mathcal{A}$ are defined as $$d_{\text{low}}(\mathcal{A})=\liminf_{r\rightarrow\infty}d_{\text{inn}}(\mathcal{A}|rB^d),$$ $$d_{\text{upp}}(\mathcal{A})=\limsup_{r\rightarrow\infty}d_{\text{out}}(\mathcal{A}|rB^d).$$

It is not hard to see that these values are independent of the choice of $O$.

The \textit{packing density} and \textit{covering density} of a convex body $C$ are given by $$\delta(C)=\sup\{d_\text{upp}(\mathcal{P}):\mathcal{P}\text{ is a packing of }\mathbb{R}^d\text{ with congruent copies of }C\},$$ $$\Theta(C)=\text{inf}\{d_\text{low}(\mathcal{C}):\mathcal{C}\text{ is a covering of }\mathbb{R}^d\text{ with congruent copies of }C\}.$$

The \textit{translational packing density} $\delta_H(C)$ and the \textit{translational covering density} $\Theta_H(C)$ are defined by taking the supremum and infimum over all packings and coverings with translates of $C$, instead of congruent copies. See \cite{packandcover} for a summary of the known bounds for the packing and covering densities.

Notice that the definitions of upper and lower density of $\mathcal{A}$ with respect to $D$ are directly tied to the Lebesgue measure, but could be readily extended to other measures. Similarly, the translates of $C$ can be interpreted as homothets of $C$ that have the same Lebesgue measure as $C$. These observations motivate the following generalization of the previous definitions.

Let $\mu$ be a  measure on $\mathbb{R}^d$. For a family $\mathcal{A}$ of sets of finite measure and a set $D$, also of finite measure, we define the \textit{inner density with respect to $\mu$} $d_{inn}(\mu,\mathcal{A}|D)$ and the \textit{outer density with respect to $\mu$} $d_{out}(\mu,\mathcal{A}|D)$ as $$d_{inn}(\mu,\mathcal{A}|D)=\frac{1}{\mu(D)}\sum_{A\in\mathcal{A},A\subset D}\mu(A),$$ $$d_{out}(\mu,\mathcal{A}|D)=\frac{1}{\mu(D)}\sum_{A\in\mathcal{A},A\cap D\neq\emptyset}\mu(A).$$

The \textit{lower density with respect to $\mu$} and \textit{upper density with respect to $\mu$} of $\mathcal{A}$ are now given by $$d_{\text{low}}(\mu,\mathcal{A})=\liminf_{r\rightarrow\infty}d_{\text{inn}}(\mu,\mathcal{A}|rB^d),$$ $$d_{\text{upp}}(\mu,\mathcal{A})=\limsup_{r\rightarrow\infty}d_{\text{out}}(\mu,\mathcal{A}|rB^d).$$  

If $\mu$ is non-$C$-degenerate and $\mu(C)>0$, then we define the \textit{homothety packing density with respect to $\mu$} and the \textit{homothety covering density with respect to $\mu$} as $$\delta_H(\mu,C)=\sup\{d_\text{upp}(\mu,\mathcal{P}):\mathcal{P}\text{ is a packing of } \mathbb{R}^d \text{ with homothets of }C\text{ of measure }\mu(C)\},$$ $$\Theta_H(\mu,C)=\text{inf}\{d_\text{low}(\mu,\mathcal{C}):\mathcal{C}\text{ is a covering of } \mathbb{R}^d \text{ with homothets of }C\text{ of measure }\mu(C)\}.$$

Given the properties of $\mu$, it is not hard to see that sets over which we take the infimum and the supremum are nonempty.

The packing and covering density can also be generalized in a natural way by considering packings and coverings with sets that are similar\footnote{Two sets $A$ and $B$ in $\mathbb{R}^d$ are similar if there exists a $\lambda>0$ such that $\lambda A$ and $B$ are congruent.} to $C$ and have fixed measure $\mu(C)$. However, all lower bounds on $\delta_H(\mu,C)$ and all upper bounds on $\Theta_H(\mu,C)$, which are one of the main focus points of this work, are obviously true for the (non-translational) packing and covering densities as well. Just as in the Lebesgue measure case, the packings and covering densities with respect to $\mu$ measure, in a sense, the efficiency of the best possible packing/covering of the measure space induced by $\mu$.

See~\cite{packandcover} for a review of the existing literature on packings and coverings and~\cite{researchproblems} for further open problems and interesting questions.

\subsection{The Besicovitch covering theorem}\label{sec:besicovitch}

The Besicovitch covering theorem extends an older result by Vitali~\cite{Vitali}. The result was first shown by Besicovitch in the planar case, and then generalized to higher dimensions by Morse~\cite{Morse_1947}, it can be stated as follows

\begin{theorem}\label{teo:besicovitch}
There is a constant $c_d$ (which depends only on $d$) with the following property: Given a bounded subset $A$ of $\mathbb{R}^d$ and a collection $\mathcal{F}$ of Euclidean balls such that each point of $A$ is the center of at least one of these balls, it is possible to find subcollections $\mathcal{F}_1,\mathcal{F}_2,\dots,\mathcal{F}_{c_d}$ of $\mathcal{F}$ such that each $\mathcal{F}_i$ consists of disjoint balls and $$A\subset\bigcup_{i=1}^{c_d}\bigcup_{B\in\mathcal{F}_i}B.$$
\end{theorem}

In fact, Morse~\cite{Morse_1947} and Bliedtner and Loeb~\cite{BliedtnerLoeb} extended the result to more general objects and normed vector spaces.  Later, F{\"u}redi and Loeb \cite{besicovitchconstant} studied the least value of $c_d$ for which the result holds.

Assume that a finite set $S\subset\mathbb{R}^d$ is such that for each point $p\in S$ there is a ball with center $p$ that covers exactly $k$ elements of $S$, then the collection of all these $|S|$ balls covers $S$. By the Besicovitch covering theorem, we can find $c_d$ subcollections, each composed of disjoint balls, whose union covers $S$. Each subcollection clearly contains at most $\frac{|S|}{k}$ balls and, thus, their union forms a covering of $S$ formed by at most $c_d\frac{|S|}{k}$ $k^-/S$-homothets of $B^d$. Since the union of the subcollections covers $S$, it contains at least $\frac{|S|}{k}$ balls, and we can find a subcollection with at least $\frac{1}{c_d}\frac{|S|}{k}$ balls, which is actually a packing formed by $k^+/S$-homothets of $B^d$. This shows that $f(B^d,k,S)=O_d(\frac{|S|}{k})$ and $g(B^d,k,S)=\Omega_d(\frac{|S|}{k})$. A careful analysis of the proof by Bliedtner  and Loeb~\cite{BliedtnerLoeb} (combined with some other geometric results), reveals that this can be extended to general convex bodies.

\subsection{VC-dimension and \texorpdfstring{$\epsilon$}{e}-nets}\label{sec:VCnets}

A \textit{set system} is a pair $\Sigma =(X,\mathcal{R})$, where $X$ is a set of base elements and $R$ is a collection of subsets of $X$. Given a set system $\Sigma =(X,\mathcal{R})$ and a subset $Y\subset X$, let $\mathcal{R}|_Y=\{Y\cap R : R\in\mathcal{R}\}$. The VC-dimension of the set system is the maximum integer $d$ for which there is a subset $Y\subset X$ with $|Y|=d$ such that $\mathcal{R}|_Y$ consists of all $2^d$ subsets of $Y$, the VC-dimension may be infinite. In a way, the VC-dimension measures the complexity of a set system, and it plays a very important role in multiple areas, such as computational geometry, statistical learning theory, and discrete geometry.

Let $\Sigma =(X,\mathcal{R})$ be a set system with $X$ finite. An $\epsilon$-\textit{net} for $\Sigma$ is a set $N\subseteq X$ such that $N\cap R\neq \emptyset$ for all $R\in\mathcal{R}$ with $|R|\geq \epsilon |X|$. A landmark result of Haussler and Welzl~\cite{VCdimension} tells us that range spaces with VC-dimension at most $d$ admit $\epsilon$-nets whose size depends only on $d$ and $\frac{1}{\epsilon}$; in fact, any random subset of $X$ of adequate size will be such an $\epsilon$-net with high probability. The precise bounds were later improved by Pach and Tardos \cite{boundsfornets}.

Given a point set $X$ and a family $\mathcal{R}$ of sets (which are not necessarily subsets of $X$), the \textit{primal set system} $(X,\mathcal{R}|_X)$ induced by $X$ and $\mathcal{R}$ is the set system with base set $X$ and $\mathcal{R}|_X=\{R\cap X\text{ }|\text{ } R\in\mathcal{R}\}$. If $X$ is finite, a \textit{weak} $\epsilon$-\textit{net} for the range space $(X,\mathcal{R}|_X)$ is a set of elements $W\subset\bigcup_{R\in\mathcal{R}R}$ such that $W\cap R\neq \emptyset$ for all $R\in\mathcal{R}$ with $|R|_X|\geq \epsilon |X|$. Weak $\epsilon$-nets have been particularly studied in geometric settings, where $X$ is a set of points and the elements of $\mathcal{R}$ are geometric objects; and this is also the setting that we care about here. The most famous result in the subject asserts the existence of a weak $\epsilon$-net whose size depends only on $d$ and $\epsilon$ for any primal set system induced by a finite set of points and the convex subsets of $\mathbb{R}^d$, the best known upper bounds on the size of such a net are due to Rubin \cite{RubinHigDim,RubinWeakNets}. Weak epsilon nets can also be defined for finite measures: if $\mu$ is finite and $\mathcal{R}$ is a family of sets in $\mathbb{R}^d$, a weak $\epsilon$-net for the pair $(\mu,\mathcal{R})$ consists of a collection $W$ of points in $\mathbb{R}^d$ such that $W\cap R\neq \emptyset$ for all $R\in\mathcal{R}$ with $\mu(R)\geq \epsilon\mu(\mathbb{R}^d)$.

We refer the reader to~\cite{surveynets} for a survey on $\epsilon$-nets and other similar concepts.

\subsection{Delaunay triangulations }\label{sec:delaunay}

Given a finite point set $S\subset\mathbb{R}^2$, the \textit{Delaunay graph} $D(S)$ is the embedded planar graph with vertex set $S$ in which two vertices are adjacent if an only if there is an Euclidean ball that contains those two points but no other point of $S$. It is not hard to check that $D(S)$ is indeed planar and that, as long as no four points lie on a circle and no three belong to the same line, $D(S)$ will actually be a triangulation\footnote{An embedded planar graph with vertex set $S$ is a \textit{triangulation} if all its bounded faces are triangles and their union is the convex hull of $S$.}. 

Delaunay graphs have a natural generalization which arises from considering general convex bodies instead of balls. The \textit{Delaunay graph of $S$ with respect to $C$}, which we denote by $D_C(S)$, is the embedded planar graph with vertex set $S$ and an edge between two vertices if an only if there is an homothet of $C$ that covers those two points but no other point of $S$. If $C$ is strictly convex and has smooth boundary, and $S$ is in $C$-general position and does not contain three points on the same line, then $D_C(S)$ will again be a triangulation. The edges of $D_C(S)$ encode the pairs of points of $S$ that can be covered using a $2^-/S$-homothet of $C$ and, thus, finding an optimal cover with $2^-/S$-homothets is equivalent to finding the largest possible matching in $D_C(S)$. 

It is good to keep in mind that Delaunay graphs can be defined analogously in higher dimensions, even if we will only really need them in the planar case.  

Many properties of generalized Delaunay triangulations can by found in Cano's PhD dissertation \cite{generalizeddelaunay}.

\subsection{Previous related work}\label{sec:previouswork}
The functions $f$ and $g$ have been indirectly studied in some particular cases. The first instance of this that we know of appeared in a paper by Szemerédi and Trotter~\cite{combdistprojective}, who obtained a lemma that implies a bound of $g(C,k,S)=\Omega(\frac{|S|}{k})$ in the case that $C$ is a square in the plane; they applied this result to a point-line incidence problem.

Dillencourt~\cite{delaunaytoughness} studied the largest matching that can be obtained in a point set using disks; in our setting, this is actually equivalent to the $k=2$  case of the covering problem. Dillencourt showed that all planar Delaunay triangulations (with respect to disks) are $1$-tough\footnote{Given a positive real number $t$, a graph $G$ is $t$-\textit{tough} if in order to split it into any number $k \geq 1$ of connected components, we need to remove at least $tk$ vertices.} and thus, by Tutte's matching theorem, contain a matching of size $\lfloor\frac{|S|}{2}\rfloor$. Ábrego et al. \cite{matchingsquares} obtained a similar result for squares; they essentially proved that, as long as no two points lie on the same vertical or horizontal line, the Delaunay triangulation with respect to an axis aligned square contains a Hamiltonian path and, as a consequence, a matching of size $\lfloor\frac{|S|}{2}\rfloor$. These results immediately translate to $f(C,2,S)\leq\lceil\frac{|S|}{2}\rceil$ whenever $C$ is a disk or a square (and $S$ has the required properties), this bound is obviously optimal. Panahi et al. \cite{trianglematchingsfirst} and Babu et al. \cite{trianglematchings} studied the problem for equilateral triangles (their results actually hold for any triangle, as can be seen by applying an adequate affine transformation), it was shown in the second of these papers that as long as $S$ is in general position the corresponding Delaunay graph must admit a matching of size at least $\lceil\frac{|S|-1}{3}\rceil$ and that this is tight. Ábrego et al. \cite{matchingsquares} also studied \textit{strong matchings} for disks and squares, which are interior disjoint collections of homothets, each of which covers exactly two points of the set, their results imply that $g(C,2,S)\geq\lceil\frac{|S|-1}{8}\rceil$ if $C$ is a disk and $g(C,2,S)\geq\lceil\frac{|S|}{5}\rceil$ if $C$ is a square, again under some mild assumptions on $S$. The bound for squares was improved to $g(C,2,S)\geq\lceil\frac{|S|-1}{4}\rceil$ by Biniaz et al. in \cite{strongmatchings}, where they also showed that $g(C,2,S)\geq\lceil\frac{n-1}{9}\rceil$ in the case that $C$ is an equilateral triangle and presented various algorithms for computing large strong matchings of various types. In a similar vein, large matchings in Gabriel graphs\footnote{The \textit{Gabriel graph} of a planar point set $S$ is the graph in which two points $p,q\in S$ are joined by an edge if an only if the disk whose diameter is the segment from $p$ to $q$ contains no other point of $S$.} and strong matchings with upward and downward equilateral triangles are treated in \cite{gabrielmatching,strongmatchings}.

Bereg et al.~\cite{matchingnp} considered matchings and strong matchings of points using axis aligned rectangles and squares. They provided various algorithms for finding large such matchings and showed that deciding if a point set has a strong perfect matching using squares (i.e. deciding if $g(C,2,S)=\frac{|S|}{2}$ in the case that $C$ is a square) is $NP$-hard.

\section{Results}\label{chap:results}

\subsection{Overview of Section \ref{chap:cover}}\label{sec:overview4}
In Section~\ref{sec:weaknets} we use a simple technique by Kulkarni and Govindarajan~\cite{weaknets} to construct a weak $\epsilon$-net of size $O_d(\frac{1}{\epsilon})$ for any primal range space (on a finite base set of points $S$) induced by the family $\mathcal{H}_C$ of all homothets of a convex body $C$. This result follows too from the known bounds on the Hadwiger-Debrunner $(p,q)$-problem for homothets (see ~\cite{pq-problem}), but our proof is short and elementary, and it also yields an analogous result for finite measures. We remark that Naszódi and Taschuk~\cite{infiniteVC} showed that $(\mathbb{R}^d, \mathcal{H}_{C})$ may have infinite VC-dimension for $d\geq 3$, so there might be no small (strong) $\epsilon$-net for $(S, \mathcal{H}_{C}|_S)$. For $d=2$, however, any range space induced by pseudo-disks (and, in particular, $(S, \mathcal{H}_{C}|_S)$), admits an $\epsilon$-net of size $O(\frac{1}{\epsilon})$~\cite{newexistenceproofs,shallowcell}.

In Section \ref{sec:covering}, we use the result on weak $\epsilon$-nets to show that, under some mild assumptions, $f(C,k,S)=O_d(\frac{|S|}{k}),f(C,k,\mu)=O_d(\frac{\mu(\mathbb{R}^d)}{k})$. The proof does not make use of the Besicovitch covering theorem (see Section \ref{sec:besicovitch}), and it will provide us with a scheme for designing one of the algorithms discussed in Section \ref{chap:computational}.

The bound for measures is then applied in Section \ref{sec:coverdensity} to prove that if $\mu$ is non-$C$-degenerate, $\mu(C)>0$ and $\mu(\mathbb{R}^d)=\infty$, then the translational covering density $\Theta_H(\mu,C)$ is bounded from above by a function of $d$. It is easy to see that $\Theta_H(\mu,C)$ is infinite for finite measures, so the $\mu(\mathbb{R}^d)=\infty$ condition is essential.

\subsection{Overview of Section \ref{chap:pack}}\label{sec:overview5}

In Section \ref{sec:pack} we prove that, under the same conditions that allowed us to obtained an upper bound for $f$, $g(C,k,S)=\Omega_d(\frac{|S|}{k}),g(C,k,\mu)=\Omega_d(\frac{\mu(\mathbb{R}^d)}{k})$. The proof relies on some properties of collections of homothets which intersect a common homothet; this is very similar to the study of $\tau$\textit{-satellite configurations} in the proof of  \cite{BliedtnerLoeb,Morse_1947}.

Similar to the covering case, the bound on $g$ is then utilized in Section \ref{sec:packdensity} to prove that if $\mu$ is non-$C$-degenerate and $\mu(\mathbb{R}^d)>\mu(C)>0$, then the translational packing density $\Theta_H(\mu,C)$ is bounded from below by a function of $d$. The $\mu(\mathbb{R}^d)>\mu(C)$ condition is clearly necessary.

\subsection{Overview of Section \ref{chap:computational}}\label{sec:overview6}

Given $C\subset\mathbb{R}^d$ and a positive integer $k$, let  $C$-$k$-COVER denote the optimization problem that consists of determining, given an instance point set $S\subset\mathbb{R}^d$, the least integer $m$ such that $S$ can be covered by $m$ $k^-/S$-homothets of $C$. Similarly, the problem $C$-$k$-PACK consists of finding the largest $m$ such that there is a packing composed of $m$ $k^+/S$-homothets of $C$.  

Section \ref{sec:algorithms} is devoted to the description of polynomial time algorithms for approximating $C$-$k$-COVER and $C$-$k$-PACK up to a multiplicative constant in the case that $C$ is a disk. The proofs are based on the ideas developed in sections \ref{sec:weaknets}, \ref{sec:covering} and \ref{sec:pack}.

There has been extensive research regarding the complexity of geometric set cover problems, and a variety of these have been shown to be NP-complete, see~\cite{optpackcovernp} for one of the first works in this direction. As mentioned in Section \ref{sec:previouswork}, Bereg et al.~\cite{matchingnp} proved that when $C$ is a square it is NP-hard to decide if $g(C,2,S)=\frac{|S|}{2}$; this implies, in particular, that $C$-$2$-COVER is NP-hard for squares. As long as we are capable of computing $D_C(S)$ in polynomial time (which is the case for hypercubes, balls and any other convex body which can be described by a bounded number of algebraic inequalities), $f(C,2,S)$ can be computed, also in polynomial time, by applying any of the known algorithms for finding the largest possible matching in a given graph. However, in Section~\ref{sec:complexity} we show that if $C$ is a square and $k$ is a multiple of $4$, then deciding if $f(C,k,S)=\frac{|S|}{k}$ is NP-hard. Unfortunately, our proof is not very robust in the sense that it depends heavily on the fact that $C$ is a square and that $S$ is not required to be in general position.  

\subsection{Overview of Section \ref{chap:matching}}\label{sec:overview7}

As mentioned in Section~\ref{sec:previouswork}, Dillencourt~\cite{delaunaytoughness} showed that the Delaunay triangulation (with respect to disks) of a point set $S\subset\mathbb{R}^2$ with no three points on the same line and no four points on the same circle is $1$-tough. Biniaz~\cite{simpletough} later gave a simpler proof of this result. 

In Section \ref{sec:tough} we extend the technique of Biniaz to show that, under some assumptions on $C$ and $S$, $D_C(S)$ is almost $t$-tough, where $t$ depends on how fat $C$ is (or, rather, how fat it can be made by means of an affine transformation). This result is then applied, again in similar fashion to \cite{simpletough}, in Section \ref{sec:matchings} to bound $f(C,2,S)$. Using a well-known result by Nishizeki and Baybars \cite{planarmatchings} on the size of the largest matchings in planar graphs, we also obtain a weaker bound that holds in greater generality.

\section{Covering}\label{chap:cover}

\subsection{Small weak \texorpdfstring{$\epsilon$}{e}-nets for homothets\label{sec:weaknets}}

The purpose of this section is to prove the following result about weak $\epsilon$-nets. 

\begin{theorem} \label{teo:nets}
Let $C\subset\mathbb{R}^d$ be a convex body and denote the family of all homothets of $C$ by $\mathcal{H}_{C}$. Then, for any finite set $S\subset\mathbb{R}^d$ and any $\epsilon>0$, $(S, \mathcal{H}_{C}|_S)$ admits a weak $\epsilon$-net of size $O_d(\frac{1}{\epsilon})$, where the hidden constant depends only on $d$. Similarly, for any $C$-nice measure $\mu$, $(\mu,\mathcal{H}_{C})$ admits weak $\epsilon$-net of size $O_d(\frac{1}{\epsilon})$.
\end{theorem}

The simple lemma below will provide us with the basic building blocks for constructing the weak $\epsilon$-net.

\begin{lemma}\label{teo:hit}
There is a constant $c_1=c_1(d)$ with the following property: Given a convex body $C\subset\mathbb{R}^d$, there is a finite set $P_C\subset\mathbb{R}^d$ of size at most $c_1$ that hits every homothet $C'$ of $C$ with $C'\cap C\neq\emptyset$ and homothety coefficient at least $1$.
\end{lemma}

\begin{proof}
Let $T$ be an affine transformation as in Lemma~\ref{teo:fat}. We begin by showing the result for $C_T=T(C)$. Every homothet $C_T'$ with $C_T'\cap C_T\neq\emptyset$ and coefficient at least $1$ contains a translate $C_T''$ of $C_T$ with $C_T''\cap C_T\neq\emptyset$; this translate satisfies $C_T''\subseteq dB^d+2dB^d\subset [-3d,3d]^d$. On the other hand, $B^d\subset C_T$, so $C_T''$ must contain a translate of an axis parallel $d$-hypercube of side $\frac{2}{\sqrt{d}}$. Now it is clear that we may take $P_{C_T}$ to be the set of points from a $\frac{2}{\sqrt{d}}$ grid\footnote{By a $\frac{2}{\sqrt{d}}$ grid we mean an axis parallel $d$-dimensional grid with separation $\frac{2}{\sqrt{d}}$ between adjacent points.} that lie in the interior of $[-3d,3d]^d$, and this grid may be chosen so that $|P_{C_T}|\leq (3d^{3/2})^d$. Setting $c_1(d)=(3d^{3/2})^d$ and $P_C=T^{-1}(P_{C_T})$ yields the result.
\end{proof}

Notice that the value $1$ plays no special role in the proof, the result still holds (with a posssibly larger $c_1$) if we wish for $P_C$ to hit every homothet whose coefficient is bounded from below by a positive constant. The construction used in the proof has the added benefit that it allows us to compute $P_{C}$ in constant time (for fixed $d$), so long as we know $T$.

Using some known results, it is possible to obtain better bounds for $c_1$. In fact, a probabilistic approach by Erd\H{o}s and Rogers \cite{erdosrogers} (see also \cite{hittingtranslates}) shows that we can take \[c_1(d)\leq3^{d+1}2^d\frac{d}{d+1}d(\log d+\log\log d+4)\] for all large enough $d$. See \cite{Hellyandrelatives} for some earlier bounds on $c_1(d)$.

Next, we prove Theorem~\ref{teo:nets}.

\begin{proof}
We show that $(S,\mathcal{H}_C|_S)$ admits a small weak $\epsilon$-net, the proof for $(\mu,\mathcal{H}_C)$ is analogous. The weak $\epsilon$-net $W$ is constructed by steps. Consider the smallest homothet $C'$ of $C$ which contains at least $\epsilon |S|$ points of $S$ and add the elements of the set $P_{C'}$, given by Lemma~\ref{teo:hit}, to $W$. Now, we forget about the points covered by $C'$ and repeat this procedure with the ones that remain until there are less than $\epsilon |S|$ points left. Since we pick at most $c_1$ points at each step, $|W|\leqslant c_1\frac{1}{\epsilon}$, so all that is left to do is show that $W$ is a weak $\epsilon$-net for $(S,\mathcal{H}_{C}|_S)$.

Let $C_{1}$ be an homothet with $C_{1}\cap S\geqslant\epsilon |S|$ and consider, along the process of constructing $W$, the first step at which the taken homothet contains at least one element of $S\cap C_{1}$, this homothet will be called $C_{2}$. Clearly, $C_{1}$ and $C_{2}$ have nonempty intersection and, since none of the points in $C_1$ had yet been erased when $P_{C_2}$ was added to $W$, $C_{1}$ is not smaller than $C_{2}$. It follows that $C_{1}$ contains at least one point of $P_{C_2}\subset W$, as desired. 
\end{proof}

As mentioned in the introduction, the technique from the last paragraph was first used by Kulkarni and Govindarajan~\cite{weaknets} to show that primal set systems induced by hypercubes and disks admit weak $\epsilon$-nets of size $O(\frac{1}{\epsilon})$.

We remark that if $c$ is a constant then it suffices to take, at each step, an homothet $C'$ that contains at least $\epsilon |S|$ points and its coefficient is at most $c$ times larger than the coefficient of the smallest homothet with that property, and then add to $W$ the set given by Lemma \ref{teo:hit} when $1$ is substituted by $1/c$. This observation will be important in Section \ref{chap:computational}.

\subsection{Covering finite sets and measures}\label{sec:covering}

At last, we state the main result about the asymptotic behavior of the function $f$ defined in Section \ref{sec:packcover}.

\begin{theorem}\label{teo:fbound}
Let $C\subset\mathbb{R}^d$ be a convex body. Then, for any positive integer $k$ and any non-$\frac{k}{2}/C$-degenerate finite set of points $S\subset\mathbb{R}^d$, we have that $f(C,k,S)=O(\frac{|S|}{k})$, where the hidden constant depends only on $d$. Similarly, for any positive real number $k$ and any $C$-nice measure, $f(C,k,\mu)=O(\frac{\mu(\mathbb{R}^d)}{k})$.
\end{theorem}

Again, we start by proving the result for point sets and then discuss the minor adaptations that must be made when working with measures.

As was essentially done in the proof of Lemma~\ref{teo:hit}, we may and will assume that $B^d\subseteq C\subseteq dB^d$. 

\begin{observation}\label{teo:tammes}
For any $d$ and any positive real $r$, there is a constant $c(d,r)$ with the following property: every set of points on $\mathbb{S}^{d-1}$ which contains no two distinct points at distance less than $r$ has at most $c(d,r)$ elements.
\end{observation}

\begin{proof}
Obvious. A straightforward $(d-1)$-volume counting argument yields $$c(d,r)<\frac{\text{vol}_{d-1} (\mathbb{S}^{d-1})}{\text{vol}_{d-1}(B^{d-1})r^{d-1}}.$$
\end{proof}

Determination of the optimal values of $c(d,r)$ is often referred to as the Tammes problem. Exact solutions are only known in some particular cases, see~\cite{tammes} for some recent progress and further references.

The simple geometric lemma below will allow us construct the desired covering.

\begin{lemma}\label{teo:neighborhoodcover}
Let $P\subset\mathbb{R}^d$ be a (possibly infinite) bounded set and consider a collection of homothets $\lbrace C_p\rbrace_{p\in P}$ such that $C_p$ is of the form $p+\lambda C$ and $\bigcap_{p\in P} C_p\neq\emptyset$. Then there is a subset $P'$ of $P$ of size at most $c_2=c_2(d)$ such that the collection of homothets $\lbrace C_p\rbrace_{p\in P'}$ covers $P$.
\end{lemma}

\begin{proof}
Take $c_2(d)=c(d,t)$ (as in the claim above) for some sufficiently small $t=t(d)$ to be chosen later. After translating, we may assume that $O\in\bigcap_{p\in P} C_p$. We construct $P'$ by steps, starting from an empty set. At each step, denote by $N$ the supremum of the Euclidean norms of the elements of $P$ that are yet to be covered by $\lbrace C_p\rbrace_{p\in P'}$, and add to $P'$ an uncovered point with norm at least $(1-\frac{1}{10d})N$. The process ends as soon as $P\subset\bigcup_{p\in P'} C_p$, we show that this takes no more than $c_2$ steps. Suppose, for the sake of contradiction, that after some number of steps we have $|P'|>c_2$ and let $P'_{unit}=\lbrace\frac{p}{|p|}\ |\ p\in P'\rbrace$. By Observation~\ref{teo:tammes} there are two  distinct points $\frac{p_1}{|p_1|},\frac{p_2}{|p_2|}\in P'_{unit}$ (with $p_1,p_2\in P'$) at distance less than $t$ from each other. Say, w.l.o.g., that $p_1$ was added to $P'$ prior to $p_2$; it follows from the construction that $|p_1|>(1-\frac{1}{10d})|p_2|$. Since $C_{p_1}$ is $1/d$-fat and contains $O$, the ball with center $p_1$ and radius $\frac{|p_1|}{d}$ lies completely within said homothet. Now, by convexity, $C_p$ must contain a bounded cone with vertex $O$, base going trough $p_1/(1-\frac{1}{10d})$, and whose angular width depends only on $d$. It follows that if $t$ is small enough then $p_2$ lies within this cone and is thus contained in $C_{p_1}$ (see figure \ref{fig:1}). This contradicts the assumption that $p_2$ was added after $p_1$, and the result follows.
\end{proof}

\begin{figure}[!htbp]
\centering
\includegraphics[scale=0.4]{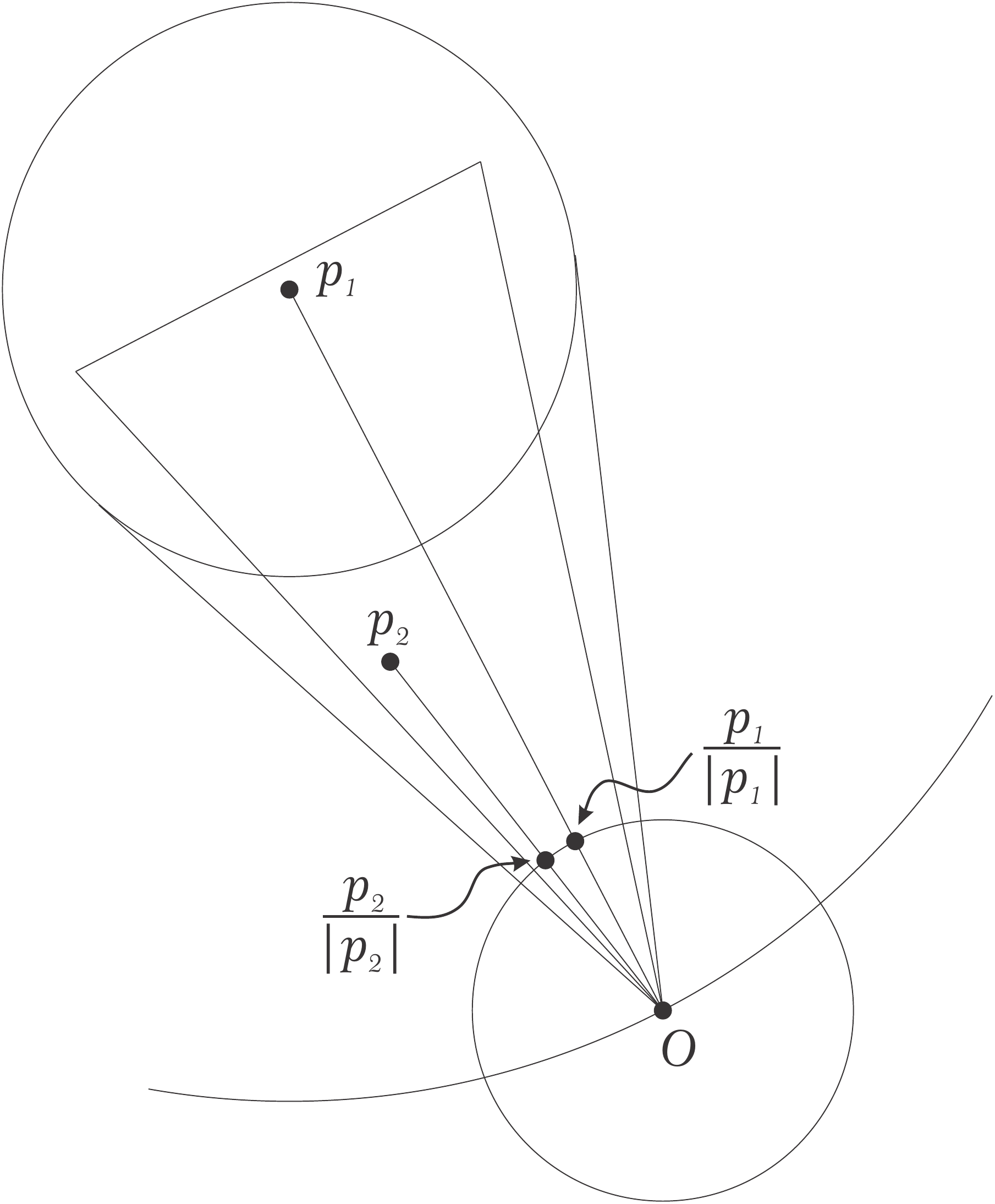}
\caption{The point $p_2$ is contained in a cone which lies completely inside $C_{p_1}$.}
\label{fig:1}
\end{figure}

We remark that the above result can easily be derived from the work of Nasz{\'o}di et al. \cite{homothetarrangements} (see also \cite{besicovitchconstant,exponentiallowerbound}).

Now we present the proof of Theorem~\ref{teo:fbound}.

\begin{proof}
We assume that $\frac{|S|}{k}\geq 1$. For every $p\in S$, let $C_{p}$ be the smallest homothet of the form $\lambda C+p$ which covers more than $\frac{k}{2}$ points of $S$ (it exists, since $C$ is closed and for any sufficiently large $\lambda$ the homothet $\lambda C+p$ covers $|S|>\frac{k}{2}$ points). Since the boundary of $C_p$ contains at most $\frac{k}{2}$ points, a slightly smaller homothet, also of the form $\lambda C+p$, will cover at least $|C_p\cap S|-\frac{k}{2}$ but at most $\frac{k}{2}$ points. It follows that $|C_p\cap S|\leqslant k$, that is, $C_p$ is a $k^-/S$-homothet. Let $C_S=\lbrace C_p\text{ }|\text{ } p\in S\rbrace$ and consider a weak $\frac{k}{2|S|}$-net $W$ for $(S,\mathcal{H}_C|_S)$ of size $O(\frac{|S|}{k/2})=O(\frac{|S|}{k})$, as given by Theorem~\ref{teo:nets}. $W$ hits every homothet of $C$ which covers at least $\frac{k}{2}$ elements of $S$ so, in particular, it hits all homothets in $C_S$. We will use Lemma~\ref{teo:neighborhoodcover} to construct the desired cover using elements of $C_S$. For each $w\in W$ let $S_w=\lbrace s\in S\ |\ w\in C_s\rbrace$. The point set $S_w$ and the homothets $\lbrace C_p\rbrace_{p\in S_w}$ satisfy the properties required in the statement of Lemma~\ref{teo:neighborhoodcover}, so there is a subset $S_w'\subset S_w$ of size at most $c_2$ such that the collection of homothets $\lbrace C_p\rbrace_{p\in S_w'}$ covers $S_w$. Let $S_C=\bigcup_{w\in W} S_w'$, we claim that the collection of $k^-/S$-homothets $\lbrace C_p\rbrace_{p\in S_C}$ covers $S$. Indeed, for every $p\in S$, $C_p$ is hit by some element of $w$, whence $s\in S_w$ and $s\in\bigcup_{p\in S_w'} C_p\subset\bigcup_{s\in S_C} C_p$. Furthermore, $|S_C|\leqslant c_2|W|=O_d(\frac{|S|}{k})$, as desired.

Now, suppose that $\mu$ is $C$-nice and $K$ is a ball with $\mu(K)=\mu(\mathbb{R}^d)\geq k$. For every $p\in K$, let $C_p$ be an homothet of the form $\lambda C+p$ which has measure $k$ (again, it exists, since $\mu$ is not-$C$-degenerate and $\mu(K)\geq k$). Let $C_\mu=\lbrace C_p\text{ }|\text{ } p\in K\rbrace$ and consider a weak $\frac{k}{\mu(\mathbb{R}^d)}$-net for $(\mu,\mathcal{H_C})$ of size $O(\frac{\mu(\mathbb{R}^d)}{k})$. From here, we can follow the argument in the above paragraph to find a collection of $O_d(\frac{\mu(\mathbb{R}^d)}{k})$ $k^-/\mu$-homothets (in fact, of homothets of measure exactly $k$) which cover $K$. This concludes the proof.
\end{proof}

We remark that the result still holds if, instead of being non-$\frac{k}{2}/C$-degenerate, $S$ is non-$tk/C$-degenerate for some fixed $t\in(0,1)$. In fact, this condition can be dropped altogether in the case that $C$ is strictly convex. The implicit requirement that $\mu$ be non-$C$-degenerate could also be weakened, all that is needed is for no boundary of an homothet to have measure larger than $tk$ (again, for fixed $t\in(0,1)$).

The proof of Theorem~\ref{teo:fbound} (as well as Theorem \ref{teo:nets}) extends almost verbatim to weighted point sets. In the weighted case, the homothets are allowed to cover a collection of points with total weight at most $k$, and the result tells us that, as long as no boundary of an homothet contains points with total weight larger than $\frac{k}{2}$, $S$ can be covered using $O_d(\frac{w(S)}{k})$ such homothets, where $w(S)$ denotes the total weight of the points in $S$.

\subsection{Generalized covering density}\label{sec:coverdensity}

\begin{theorem}\label{teo:coverdensity}
Let $C\subset\mathbb{R}^d$ be a convex body and $\mu$ a non-$C$-degenerate measure such that $\mu(C)>0$ and $\mu(\mathbb{R})=\infty$. Then $\delta_H(\mu,C)$ is bounded from above by a function of $d$.
\end{theorem}

\begin{proof}
For any Borel set $K\subset\mathbb{R}^d$ the \textit{restriction of $\mu$ to $K$}, $\mu|_K$, is defined by $\mu|_K(X)=\mu(X\cap K)$. Notice that if $K$ is bounded then $\mu|_K$ is $C$-nice.

At a high level, our strategy consists of choosing an infinite sequence of positive reals, $\lambda_0<\lambda_1<\lambda_2<\dots$, and constructing covers with homothets of measure $\mu(C)$ of each of the bounded regions $\lambda_0B^d$, $\lambda_1B^d\backslash\lambda_0B^d$, $\lambda_2B^d\backslash\lambda_1B^d,\dots$ using Theorem \ref{teo:fbound} so that the union of these covers has bounded lower density with respect to $\mu$. To be entirely precise, $\lambda_{i+1}$ will not be chosen until after the cover of $\lambda_{i}B^d\backslash\lambda_{i-1}B^d,\dots$ has been constructed. The main difficulty that arises is that, after applying Theorem \ref{teo:fbound} to the restriction of $\mu$ to a bounded set, some of the homothets in the resulting cover may have measure (with respect to $\mu$) larger than $\mu(C)$. Below, we describe a process that allows us to circumvent this issue. Here, the importance of defining $d_\text{low}$ as we did (back in Section \ref{sec:packcover}) will be clear. 

Choose $\lambda_0>0$ such that $\mu(\lambda_0B^d)\geq\mu(C)$ and set $\lambda_0B^d=\lambda_0B^d$. Theorem \ref{teo:fbound} tells us that $f(C,\frac{\mu(C)}{2},\mu|_{\lambda_0B^d})\leq c_{f,d}\frac{\mu(\lambda_0B^d)}{\mu(C)}$, so $\lambda_0B^d$ can be covered using no more than $c_{f,d}\frac{\mu(\lambda_0B^d)}{\mu(C)}$ homothets of $C$ which have measure at most $\frac{\mu(C)}{2}$ with respect to $\mu|_{\lambda_0B^d}$. In fact, if all of them were $\mu(C)^-/\mu$-homothets we could apply a dilation to each so that every one had measure $\mu(C)$ with respect to $\mu$. The following lemma shows that any homothet of the cover whose measure is too large with respect to $\mu$ can be substituted by a finite number of $\mu(C)^-/\mu$-homothets which are not completely contained in $\lambda_0B^d$.

\begin{lemma}\label{teo:fixbadhom}
Let $B\subset\mathbb{R}^d$ be a ball with $\mu(B)\geq\mu(C)$ and $C'$ be an homothet of $C$ such that $\mu|_{B}(C')<\mu(C)$ but $C'\not\subset B$. Then $C'\cap B$ can be covered by a finite collection of $\mu(C)^-/\mu$-homothets of $C$, none of which is fully contained in $B$.
\end{lemma}

\begin{proof}
Of course, we may assume that $C'\cap B\neq\emptyset$ and $\mu(C')>\mu(C)$. Let $C''$ be an homothet with $\mu|_B(C')<\mu|_B(C'')<\mu(C)$ that results from applying dilation to $C'$ with center in its interior; clearly, $C'\subsetneq C''$ and $\mu(C'')>\mu(C)$. Now, let $\overline{B}$ denote the closure of $B$ and, for each $p\in C'\cap\overline{B}$, consider an homothet $C_p$ with $\mu(C_p)=\mu(C)$ that is obtained by applying a dilation to $C''$ with center $p$. Since $\mu(C'')>\mu(C)$ and $p$ lies in the interior of $C''$, $C_p\subsetneq C''$ and $p$ belongs to the interior of $C_p$ (see figure \ref{fig:2}). We claim that $C_p$ is not fully contained in $B$. Indeed, if it were, we would have $C_p\subset B\cap C''$, but $\mu(B\cap C'')=\mu|_B(C'')<\mu(C)$, which contradicts the choice of $C_p$. Thus, for each point $p\in C'\cap\overline{B}$, $C_p$ has measure $\mu(C)$ with respect to $\mu$, it is not completely contained in $B$, and it covers an open neighborhood of $p$. The result now follows from the fact that $C'\cap\overline{B}$ is compact.
\end{proof}

\begin{figure}[!htbp]
\centering
\includegraphics[scale=0.42]{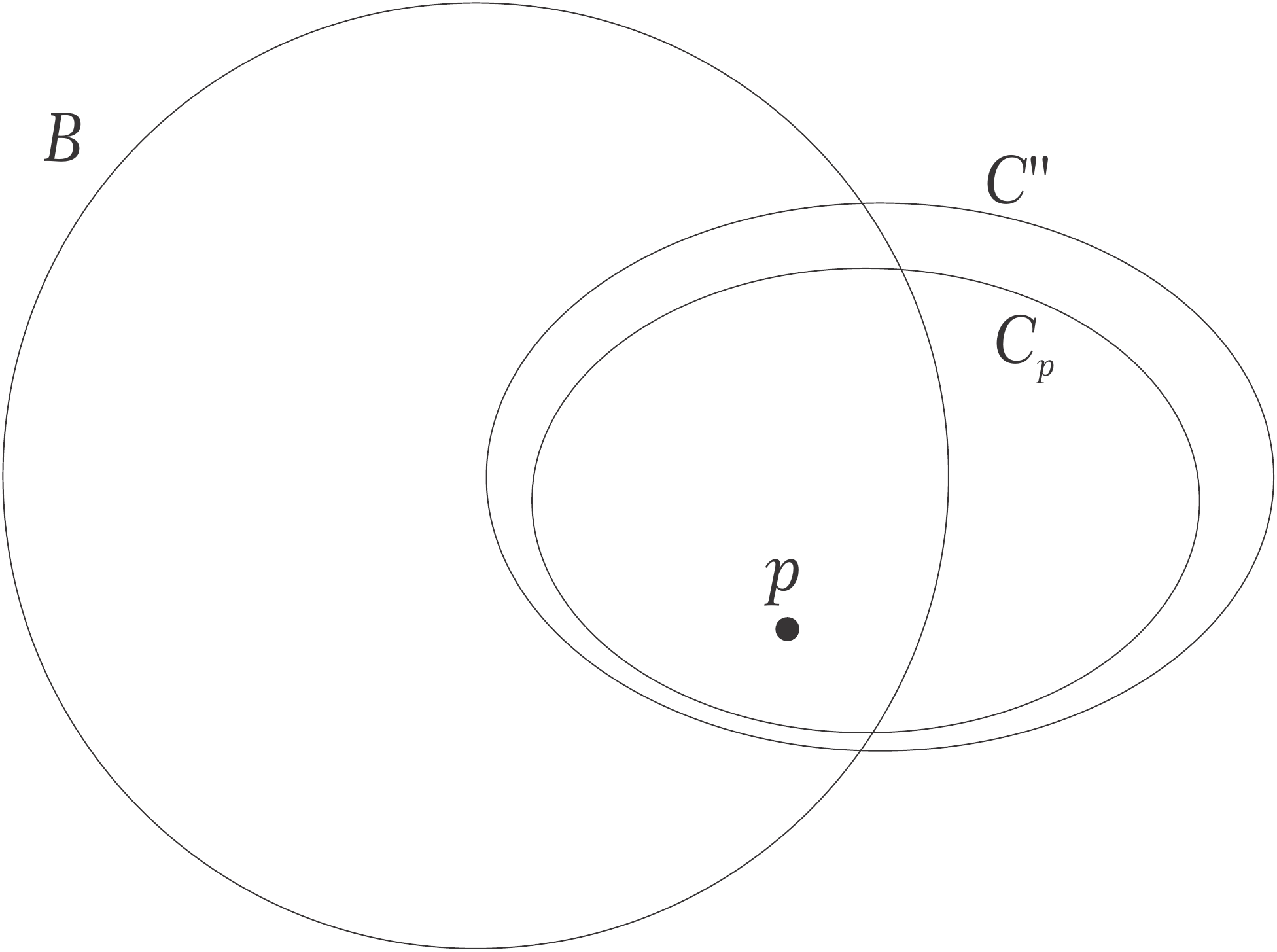}
\caption{Configuration in the proof of Lemma \ref{teo:fixbadhom}.}
\label{fig:2}
\end{figure}

Apply Lemma \ref{teo:fixbadhom} (with $B=\lambda_0B^d$) to each of the aforementioned homothets and then enlarge each homothet in the cover until its measure with respect to $\mu$ is $\mu(C)$. This way, we obtain a finite cover $\mathcal{F}_0$ of $\lambda_0B^d$ by homothets of measure $\mu(C)$ with respect to $\mu$, of which at most $c_{f,d}\frac{\mu(\lambda_0B^d)}{\mu(C)}$ are fully contained in $\lambda_0B^d$. 

Now, suppose that $\lambda_0<\lambda_1<\dots<\lambda_t$ have already been chosen so that there is a finite family $\mathcal{F}_t$ of homothets of measure $\mu(C)$ with respect to $\mu$ that covers $\lambda_iB^d$ and has the following property:  at most $2c_{f,d}\frac{\mu(\lambda_iB^d)}{\mu(C)}$ of the homothets are fully contained in $\lambda_iB^d$ for every $i\in\{0,1,\dots,t\}$.

Chose $\lambda_{t+1}$ so that $2\lambda_t<\lambda_{t+1}$ and $\mu(\lambda_{i+1}B^d)\geq\frac{\mu(C)|\mathcal{F}_t|}{c_{f,d}}$ (the condition $\mu(\mathbb{R}^d)=\infty$ is crucial here). By Theorem \ref{teo:fbound}, $f(C,\frac{\mu(C)}{2},\mu|_{\lambda_{t+1}B^d})\leq c_{f,d}\frac{\mu(\lambda_{t+1}B^d)}{\mu(C)}$; consider a cover that achieves this bound. Again by Lemma \ref{teo:fixbadhom}, each homothet in the cover with measure larger than $\mu(C)$ with respect to $\mu$ can be substituted by a finite collection of homothets of measure at most $\mu(C)$ which are not fully contained in $\lambda_{t+1}B^d$, so that the homothets still cover $\lambda_{t+1}B^d$. After having carried out these substitutions, we enlarge each homothet in the cover so that it has measure $\mu(C)$ with respect to $\mu$ and then remove all homothets which are fully contained in $\lambda_tB^d$. The resulting family of homothets, which we denote by $F_{t+1,\text{outer}}$, covers $\lambda_{t+1}B^d\backslash\lambda_tB^d$ and contains at most $c_{f,d}\frac{\mu(\lambda_{t+1}B^d)}{\mu(C)}$ homothets that lie completely inside $\lambda_{t+1}B^d$. Let  $\mathcal{F}_{t+1}=\mathcal{F}_t\cup\mathcal{F}_{t+1,\text{outer}}$. $F_{t+1}$ is a cover of $\lambda_tB^d\cup\lambda_{t+1}B^d\backslash\lambda_tB^d=\lambda_{t+1}B^d$ that consists of homothets of measure $\mu(C)$ with respect to $\mu$. Since no element of $\mathcal{F}_{t+1,\text{outer}}$ is a subset of $\lambda_tB^d$, there are no more than $2c_{f,d}\frac{\mu(\lambda_iB^d)}{\mu(C)}$ homothets fully contained in $\lambda_iB^d$ for every $i\in\{0,1,\dots,t\}$ and there are also no more than $|\mathcal{F}_t|+c_{f,d}\frac{\mu(\lambda_{t+1}B^d)}{\mu(C)}\leq2c_{f,d}\frac{\mu(\lambda_{t+1}B^d)}{\mu(C)}$ homothets contained in $\lambda_{t+d}B^d$.

Repeating this process, we obtain a sequence $\lambda_0<\lambda_1<\dots$ that goes to infinity and a sequence $\mathcal{F}_0\subset\mathcal{F}_1\subset\dots$ of collections of homothets of measure $\mu(C)$ with respect to $\mu$. Set $\mathcal{F}=\cup_{i=0}^\infty\mathcal{F}_i$, then $\mathcal{F}$ is a cover of $\mathbb{R}^d$ with homothets of measure $\mu(C)$ and, for  $i=0,1,\dots$, we have that \[d_{inn}(\mu,\mathcal{F}|\lambda_iB^d)=\frac{1}{\mu(\lambda_iB^d)}\sum_{C'\in\mathcal{F},C'\subset \lambda_iB^d}\mu(C')\leq\frac{1}{\mu(\lambda_iB^d)}\frac{2c_{f,d}\mu(\lambda_{i}B^d)}{\mu(C)}\mu(C)=2c_{f,d},\] hence
\[d_\text{low}(\mu,\mathcal{F})=\liminf_{r\rightarrow\infty}d_{\text{inn}}(\mu,\mathcal{F}|rB^d)\leq2c_{f,d},\] and the result follows. 
\end{proof}

Just as in the previous section, the result still holds as long as no boundary of an homothet has measure larger than $tk$ for some fixed $t\in(0,1)$. Our argument can also be slightly modified to yield a cover with lower density at most $(1+\epsilon)c_{f,d}$ for any $\epsilon>0$, this implies that $\Theta_H(\mu,C)\leq c_{f,d}$ (recall that $c_{f,d}$ is the hidden constant in Theorem \ref{teo:fbound}). 

\section{Packing}\label{chap:pack}

\subsection{Packing in finite sets and measures}\label{sec:pack}

\begin{theorem}\label{teo:gbound}
Let $C\subset\mathbb{R}^d$ be a convex body. Then, for any positive integer $k$ and
any non-$\frac{k}{2}/C$-degenerate finite set of points $S\subset\mathbb{R}^d$, we have that $g(C,k,S)=\Omega_d(\frac{|S|}{k})$, where the hidden constant depends only on $d$. Similarly, for any positive real number $k$ and any $C$-nice measure, $g(C,k,\mu)=\Omega(\frac{\mu(\mathbb{R}^d)}{k})$.
\end{theorem}

Again, we assume that $B^d\subseteq C\subseteq dB^d$ and $|S|\geq k$, and we begin by proving the result for point sets.

For each $p\in S$, denote by $C_{p}$ the smallest homothet of the form $\lambda C+p$ which contains at least $k$ points of $S$. All of the $C_p$'s are $k^+/S$-homothets of $C$ and, by the assumption that $S$ is non-$\frac{k}{2}/C$-degenerate, each of them contains less than $\frac{3k}{2}$ elements of $S$. For any subset $S'\subseteq S$, let $C_{S'}=\lbrace C_{p}\text{ }|\text{ }p\in S'\rbrace$. We require the following preliminary result.

\begin{claim}\label{teo:satelliteconfig}
There is a constant $c_3=c_3(d)$ with the following property: If $S'\subset S$ and $p_0\in S'$ is such that $C_{p_0}$ is of minimal size amongst the elements of $C_{S'}$, then $C_{p_0}$ has nonempty intersection with at most $c_3k$ other elements of $C_{S'}$.
\end{claim}

\begin{proof}
After translating, we may assume that $B^d\subseteq C_{p_0}\subseteq dB^d$. For any $r\in\mathbb{R}$, the number of translates of $\frac{1}{2}B^d$ required to cover $rdB^d$ depends only on $d$ and $r$ and, by the choice of $p_0$, every one of these balls of radius $\frac{1}{2}$ contains less than $k$ points of $S'$. Hence, $|rdB^d\cap S'|\leq c_{d,r}k$.

Assume, w.l.o.g., that $p_0=O$ and let $c(d,t')$ be as in Observation \ref{teo:tammes} for some small $t'=t'(d)$ to be specified later. For each $r$, denote by $S_r'\subset S'$ the set that consists of those points $p\in S'$ such that $p\notin rdB^d$ and $C_p$ intersects $C_{p_0}$. Since $C$ is $1/d$-fat, it is not hard to see that for some large enough $r_d$ (which depends only on $d$) the following holds: if $p_1,p_2\in S_{r_d}'$ are such that $|p_1|\geq|p_2|$ and $\frac{p_1}{|p_1|},\frac{p_2}{|p_2|}$ are at distance less than $t'$, then $p_2\in C_{p_1}$. We can then proceed along the lines of the proof of Lemma \ref{teo:neighborhoodcover} to show that $S_{r_d}'$ can be covered by no more than $c(d,t')$ elements of $C_{S'}$, which yields $|S_{r_d}'|\leq\frac{3}{2}c(d,t')k$. Hence, there are at most $(c_{d,r_d}+\frac{3}{2}c(d,t'))k$ elements of $C_{S'}$ which have nonempty intersection with $C_{p_0}$, and the result follows by setting $c_3=c_{d,r_d}+\frac{3}{2}c(d,t')$.
\end{proof}

We can now prove Theorem \ref{teo:gbound}.

\begin{proof}
Let $S'\subseteq S$. We show by induction on $|S'|$ that there is a packing formed by at least $\lfloor\frac{|S'|}{c_3k}\rfloor$ elements from $C_{S'}$ (if $|S|<k$, set $C_{S'}=\emptyset$); since $C_S$ consists only of $k^+/S$ homothets of $C$, the result will follow immediately . 

Our claim is trivially true if $|S'|<c_3k$. Let $S'\subseteq S$ with $|S'|\geq c_3k$ and assume that the result holds for all subsets with less than $|S'|$ elements. Choose $p_0\in S'$ so that $C_{p_0}$ is of minimal size amongst the elements of $C_{S'}$. Let $S_{p_0}=\{p\in S'\text{ }|\text{ }C_p\cap C_{p_0}\neq\emptyset\}$ and set $S''=S'-S_{p_0}$. Since $|S''|<|S'|$, the inductive hypothesis tells us that it is possible to choose $t\geq\lfloor\frac{|S''|}{c_3k}\rfloor$ points $p_1,p_2,\dots,p_t\in S''$ so that the homothets $C_{p_1},C_{p_2},\dots,C_{p_t}$ are pairwise disjoint. By the definition of $S''$, these homothets do not intersect $C_{s_0}$, this shows that we can choose $t+1$ disjoint homothets from $C_{S'}$. By Claim \ref{teo:satelliteconfig}, $|S''|\geq |S'|-c_3k$ and hence $t\geq\lfloor\frac{|S'|}{c_3k}\rfloor-1$, which yields the result.

Now, suppose that $\mu$ is $C$-nice and $K$ is a ball with $\mu(K)=\mu(\mathbb{R}^d)>k$. For each $p\in K$, define $C_{p}$ as the smallest homothet of the form $\lambda C+p$ which has measure $k$ and, for $K'\subseteq K$, let $C_{K'}\mu=\lbrace C_{p}\text{ }|\text{ }p\in K'\rbrace$. Claim \ref{teo:satelliteconfig} can be easily adapted to work measures, which then allows us to proceed as in the previous paragraph (except that we now induct on $\mu(K')$) to prove the measure theoretic version of Theorem \ref{teo:gbound}.
\end{proof}

Similarly to Theorem \ref{teo:fbound}, the non-$\frac{k}{2}$-degeneracy condition on $S$ can be relaxed to non-$tk$-degeneracy for some fixed $t>0$, and the non-$C$-degeneracy of $\mu$ can be substituted for the weaker requirement that no boundary of an homothet has measure larger than $tk$. Again, the proof extends to suitable weighted points sets. 

In similar fashion to the proof of the Besicovitch covering theorem, it is also possible to derive Theorem \ref{teo:fbound} by adapting the technique above. Indeed, we could have defined $C_{p}$ to be the smallest homothet of the form $\lambda C+p$ that contains at least $\frac{k}{2}$ points of $S$. The proof of \ref{teo:satelliteconfig} would then yield a collection of $c_3$ $k^-/S$-homothets of $C$ that covers the set $S_{p_0}=\{p\in S\text{ }|\text{ }C_p\cap C_{p_0}\neq\emptyset\}$. We add these $O_d(1)$ homothets to the cover and add all the elements of $C_{p_0}\cap S$ to an initially empty set $P$. Now, consider $p_1\in S-S_{p_0}$ such that the size of $C_{p_1}$ is minimal and go through the same steps as before. This process is then repeated as long as $S$ is not yet fully covered. At least $\frac{k}{2}$ new elements are added to $P$ with each iteration, so the number of homothets in the final cover is no more than $\frac{2n}{k}O_d(1)=O_d(\frac{n}{k})$, as desired. The proof presented in Section \ref{chap:cover}, however, will lead to a randomized algorithm for approximating $C$-$k$-COVER in Section \ref{sec:algorithms}.

\subsection{Generalized packing density}\label{sec:packdensity}

\begin{theorem}\label{teo:packdensity}
Let $C\subset\mathbb{R}^d$ be a convex body and $\mu$ a non-$C$-degenerate measure with $\mu(C)>0$ and $\mu(\mathbb{R}^d)>\mu(C)$. Then $\Theta_H(\mu,C)$ is bounded from below by a function of $d$.
\end{theorem}

\begin{proof}
If $\mu(\mathbb{R}^d)<\infty$, the result follows readily by applying Theorem \ref{teo:gbound} to the restriction of $\mu$ to sufficiently large balls and then shrinking some homothets if necessary, so we assume that $\mu(\mathbb{R}^d)=\infty$. The strategy that we follow is similar  to the one used for Theorem \ref{teo:coverdensity}.

Choose $\lambda_0>0$ so that $\mu(\lambda_0B^d)\geq\mu(C)$. By Theorem \ref{teo:gbound}, $g(C,\mu(C),\mu|_{\lambda_0B^d})\geq c_{g,d}\frac{\mu(\lambda_0B^d)}{\mu(C)}$, so there is a collection of at least $c_{g,d}\frac{\mu(\lambda_0B^d)}{\mu(C)}$ interior disjoint $\mu(C)^+/\mu|_{\lambda_0B^d}$-homothets of $C$. Each homothet in this collection contains another homothet that has nonempty intersection with $\lambda_0B^d$ and whose measure with respect to $\mu$ is exactly $\mu(C)$. These smaller homothets form a finite packing, which we denote by $\mathcal{F}_0$. 

Assume that we have already chosen $\lambda_0<\lambda_1<\dots<\lambda_t$ so that there is a finite packing $\mathcal{F}_t$ composed by homothets of measure $\mu(C)$ and at least $c_{g,d}\frac{\mu(\lambda_iB^d)}{2\mu(C)}$ of them have nonempty intersection with $\lambda_iB^d$ for every $i\in\{0,1,\dots,t\}$.

Let $\lambda_{\mathcal{F}_t}>\lambda_t$ be such that all homothets of $\mathcal{F}_t$ are fully contained in $\lambda_{\mathcal{F}_t}B^d$. Denote the region $(\lambda_{\mathcal{F}_t}+1)B^d\backslash\lambda_{\mathcal{F}_t}B^d$ by $R$ and, for each $l>0$, let $\mu_l$ be the measure defined by \[\mu_l(X)=\mu(X\backslash(\lambda_{\mathcal{F}_t}+1)B^d)+l\ \text{vol}(X\cap R).\]

\begin{claim}\label{teo:barrier}
If $l$ is large enough, then any homothet that intersects both $\lambda_{\mathcal{F}_t}\mathbb{S}^{d-1}$ and $(\lambda_{\mathcal{F}_t}+1)\mathbb{S}^{d-1}$ has measure larger than $\frac{3}{2}\mu(C)$ with respect to $\mu_l$.
\end{claim}

\begin{proof}
The claim follows from the fact that the volume of any homothet as in the statement is bounded away from $0$. This last observation can be proven by a simple compactness argument.
\end{proof}

Let $l$ be such that the property in Claim \ref{teo:barrier} holds and choose $\lambda_{t+1}$ so that $2\lambda_t<\lambda_{t+1}$, $\lambda_{\mathcal{F}_t}<\lambda_{t+1}$ and \[\mu_l(\lambda_{t+1}B^d)\geq\frac{3\mu(\lambda_{t+1}B^d)}{4c_{g,d}} +\frac{3\text{vol}(R)}{c_{g,d}l}\] (this is possible, since we assumed that $\mu(\mathbb{R}^d)=\infty$). Theorem \ref{teo:gbound} tells us that $g(C,\frac{3}{2}\mu(C),\mu_l|_{\lambda_{t+1}B^d})\geq c_{g,d}\frac{2\mu_l(\lambda_{t+1}B^d)}{3\mu(C)}$; consider a packing by $\frac{3}{2}\mu(C)^+/\mu_l$-homothets which has at least this many elements. This packing contains at most $\frac{2\text{vol}(R)}{l\ \mu(C)}$ homothets $C'$ with $\text{vol}(C'\cap R)\ l\geq\frac{1}{2}\mu(C)$, which we remove from the collection. By the choice of $l$, none of the remaining homothets intersects $\lambda_tB^d$ and each of them has measure at least $\mu(C)$ with respect to $\mu$. Shrinking each homothet we obtain a packing $\mathcal{F}_{t+1,\text{outer}}$ formed by homothets of measure $\mu(C)$ with respect to $\mu$, and it has at least \[\frac{2c_{g,d}}{3\mu(C)}\left(\frac{3\mu(\lambda_{t+1}B^d)}{4c_{g,d}} +\frac{3\text{vol}(R)}{c_{g,d}l}\right)-\frac{2\text{vol}(R)}{l\ \mu(C)}=\frac{c_{g,d}}{2}\frac{\mu(\lambda_{t+1}B^d)}{\mu(C)}\] elements. Let $\mathcal{F}_{t+1}=\mathcal{F}_t\cup\mathcal{F}_{t+1\text{outer}}$, this is a packing with homothets of measure $\mu(C)$ with respect to $\mu$, and it contains at least $\frac{c_{g,d}}{2}\frac{\mu(\lambda_{i}B^d)}{\mu(C)}$ elements which have nonempty intersection with $\lambda_iB^d$ for each $i\in\{0,1,\dots,{t+1}\}$.

Repeating this process, we obtain a sequence $\lambda_0<\lambda_1<\dots$ that goes to infinity and a sequence $\mathcal{F}_0\subset\mathcal{F}_1\subset\dots$ of packings with homothets of measure $\mu(C)$ with respect to $\mu$. Set $\mathcal{F}=\cup_{i=0}^\infty\mathcal{F}_i$, then $\mathcal{F}$ is a packing with homothets of measure $\mu(C)$ and, for  $i=0,1,\dots$, we have that \[d_{out}(\mu,\mathcal{F}|\lambda_iB^d)=\frac{1}{\mu(\lambda_iB^d)}\sum_{C'\in\mathcal{F},C'\cap \lambda_iB^d\neq\emptyset}\mu(C')\geq\frac{1}{\mu(\lambda_iB^d)}\frac{c_{g,d}\mu(\lambda_{i}B^d)}{2\mu(C)}\mu(C)=\frac{c_{g,d}}{2},\] thus
\[d_\text{upp}(\mu,\mathcal{F})=\limsup_{r\rightarrow\infty}d_{\text{out}}(\mu,\mathcal{F}|rB^d)\geq\frac{c_{g,d}}{2},\] as desired. 
\end{proof}

Again, the result holds as long as no boundary of an homothet has measure larger than $tk$ for some fixed $t\in(0,1)$. As in the proof of Theorem \ref{teo:coverdensity}, our argument can be slightly modified to show that $\delta_H(\mu,C)\geq c_{g,d}$ (where $c_{g,d}$ is the hidden constant in Theorem \ref{teo:gbound}). 

\section{Algorithms and complexity}\label{chap:computational}

\subsection{Algorithms}\label{sec:algorithms}

In this section we describe algorithms for approximating $B^d$-$k$-COVER and $B^d$-$k$-PACK (defined in Section \ref{sec:overview6}) up to a multiplicative constant that depends on $d$. The algorithms also provide either a cover with $k^-/S$ balls or a packing  with $k^+/S$ balls with that number of elements. The algorithms essentially recreate the constructive proofs of theorems \ref{teo:gbound} and \ref{teo:fbound}.

We first present a randomized algorithm for approximating $B^d$-$k$-COVER. Given a finite point set $P\subset\mathbb{R}^d$, denote by $r_\text{opt}(P,k)$ the radius of the smallest ball that contains at least $k$ points of $P$. The following result of Har-Peled and Mazumdar \cite{k-enclosing} (see also Section 1 in \cite{geometricapproximation}) will be key.

\begin{theorem}\label{teo:approx-k-ball}
Given a set $P\subset\mathbb{R}^d$ of $n$ points and an integer parameter $k$, we can find, in expected $O_d(n)$ time, a ($d$-dimensional) ball of radius at most $2r_\text{opt}(P,k)$ which contains at least $k$ points of $P$.
\end{theorem}

\begin{theorem}\label{teo:coveralg}
Let $S\subset\mathbb{R}^d$ be a set of $n$ points. There is an algorithm that finds a covering of $S$ formed by $O_d(\frac{n}{k})$ $k^-/S$-homothets of $B^d$ in expected $O_d(\frac{n^2}{k})$ time.
\end{theorem}

\begin{proof}
By repeated applications of Theorem \ref{teo:approx-k-ball} we can find, in expected $O(\frac{n}{k}\cdot n)$ time, a sequence $B_1,B_2,\dots,B_t$ of balls and a sequence $S=S_1\supset S_2\supset\dots\supset S_{t+1}=\emptyset$ (with $t\leq\lceil\frac{2n}{k}\rceil$) such that each $B_i$ has radius at most $2r_\text{opt}(S_i,k/2)$, contains at least $k/2$ points of $S_i$ and satisfies $S_i\cap B_i=S_i-S_{i+1}$. 

For each $B_i$, we can construct a set $P_{B_i}$ as in Lemma \ref{teo:hit} in $O_d(1)$ time. The union $W$ of these $t$ sets forms a weak $\epsilon$-net for $(S,\mathcal{H}_B|_S)$ (see Theorem \ref{teo:nets}). As in the proof of Theorem \ref{teo:fbound}, for each $p\in S$ let $B_p$ be the smallest ball of the form $\lambda B^d+p$ which covers at least than $\frac{k}{2}$ points of $S$ (if $S$ is not in $\frac{k}{2}/S$-general position, we might have  to perturb $B_p$ slightly so that it contains no more than $k$ points); we do not compute any of these balls at this point in time. Each $B_p$ contains at least one element of $W$, and we can find one such $w_p\in W$ in $O_d(W)=O_d(\frac{n}{k})$ time by simply choosing from $W$ a point that minimizes the distance to $p$. This is repeated for every $p\in S$. 

For every $w\in W$, let $S_w=\{p\in S\text{ }|\text{ }w_p=w\}$. Select from $S_w$ the point $p$ that is the furthest away from $w$ and compute the ball $B_p$. This can be done in $O_d(n)$ time, even in the case that a small perturbation is required, by looking at the distances from $p$ to each other element of $S$. Add $B_p$ to the final cover, remove the points in $B_p$ from $S_w$, and repeat until $S_w$ is empty. As can be seen from the proof of Lemma \ref{teo:neighborhoodcover}, the process ends after $O_d(1)$ iterations. 

Repeat the scheme above for every $w\in W$ to obtain a cover with the desired properties. This takes $O_d(\frac{n}{k}\cdot n)$ time and, thus, the expected running time of the whole algorithm is precisely $O_d(\frac{n}{k}\cdot n)$. See Section \ref{sec:covering} for some omitted details.
\end{proof}

\begin{theorem}
Let $S\subset\mathbb{R}^d$ be a set of $n$ points. There is an algorithm that computes a packing formed by $O_d(\frac{n}{k})$ $k^+/S$-homothets of $B^d$ in $O_d(n^2)$ time.
\end{theorem}

\begin{proof}
Following the proof of Theorem \ref{teo:gbound}, for each $p\in S$ let $B_{p}$ be the smallest homothet of the form $\lambda B^d+p$ which contains at least $k$ points of $S$ (as in the previous algorithm, we might have to perturb it slightly so that it contains no more than $\frac{3k}{2}$ points) and, for $S'\subseteq S$, set $B_{S'}=\lbrace B_{p}\text{ }|\text{ }p\in S'\rbrace$. Compute all the elements of $B_{S}$ in total $O_d(n^2)$ time and find a point $p_0\in S$ such that $B_{p_0}$ is of minimal radius. Add $B_{p_0}$ to the packing. By Claim \ref{teo:satelliteconfig}, there are at most $c_3k$ points $p\in S$ such that $B_{p}$ intersects $B_{p_0}$ and, given the radius of each $B_{p}$, we can compute in linear time the set $S_{p_0}\subset S$ formed by all of these points. Now, we find a point $p_1\in S-S_{p_0}$ such that $B_{p_1}$ is of minimal radius, add it to the packing, and repeat the process above for as long as possible. At the end, we get a packing composed of $\Omega_d(\frac{n}{k})$ balls which contain at least $k$ points of $S$. Each of the (at most) $\frac{n}{k}$ iterations takes $O_d(n)$ time, so the running time of the algorithm is dominated by the $O_d(n^2)$ time that it takes to compute the elements of $B_{S}$.
\end{proof}

In the same way that the proof of Theorem \ref{teo:gbound} can be adapted to obtain an upper bound for $f$ (see the last paragraph of Section \ref{sec:covering}), we can also modify the algorithm above to get the following result.

\begin{theorem}
Let $S\subset\mathbb{R}^d$ be a set of $n$ points. There is an algorithm that computes, in $O_d(n^2)$ time, a cover of $S$ formed by $O_d(\frac{n}{k})$ $k^-/S$-homothets of $B^d$.
\end{theorem}

\subsection{Complexity}\label{sec:complexity}

As mentioned in Section \ref{sec:previouswork}, Bereg et al. \cite{matchingnp} showed if $C$ is a square then deciding whether $g(C,2,S)=\frac{|S|}{2}$ is NP-hard. We prove a similar result for $C$-$k$-COVER.

\begin{theorem}\label{teo:complexity}
Let $C$ be a square and $k$ a positive multiple of $4$. Then $C$-$k$-COVER is NP-hard. In fact, it is NP-hard to determine whether $f(C,k,S)=\frac{|S|}{k}$ or not.
\end{theorem}

\begin{proof}
Suppose that $C$ is a square. We provide a polynomial time reduction from $3$-SAT\footnote{$3$-SAT consists of determining the satisfability of a Boolean formula in conjunctive normal form where each clause has three variables. $3$-SAT is well-known to be NP-complete.} to $C$-$4$-COVER. The construction can easily be adapted to work for any $k$ multiple of $4$.

Suppose we are given an instance of $3$-SAT. To each variable we will assign a collection of points with integer coordinates which form a sort of loop; the number of points in each of these loops will be a multiple of $4$. For each clause, there will be a couple of smaller loops formed too by integer points; the number of points in each of these two loops will be even, but not a multiple of $4$.  The total number of points will thus be a multiple of $4$, say, $4m$. We will call a square \textit{good} if it covers exactly $4$ points. The goal is to construct the loops in such a way that the Boolean formula is satisfiable if and only if the points can be covered by $m$ good squares. Such a collection of squares will be referred to as a \textit{good cover}. Note that in a good cover each point is covered by exactly one square. For an overview of the construction, see figure \ref{fig:3}. 

\begin{figure}[!htbp]
\centering
\includegraphics[scale=0.75]{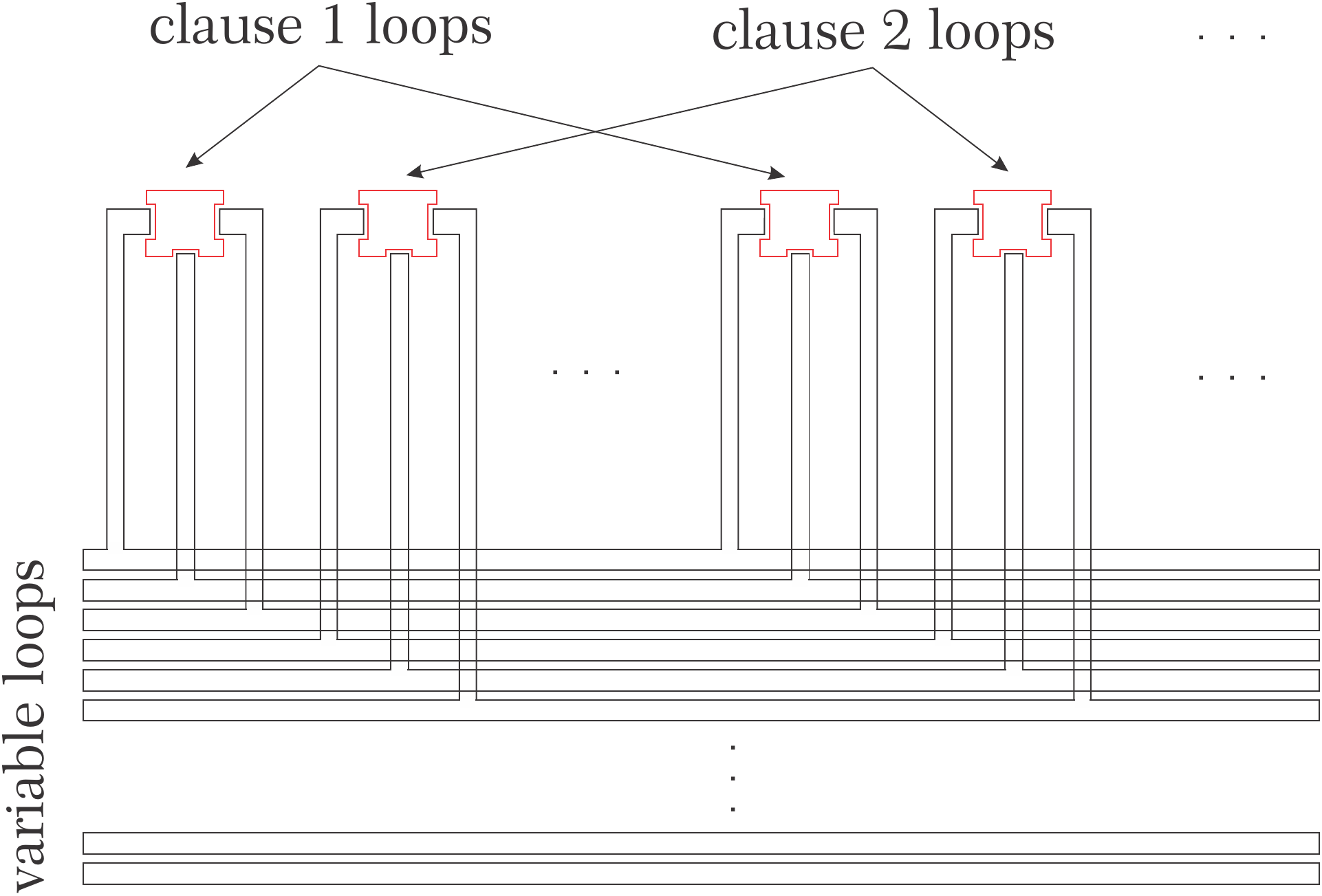}
\caption{Overview the layout of the variable loops (black) and clause loops (red).}
\label{fig:3}
\end{figure}

At each crossing between two variable loops the points are arranged as in figure \ref{fig:4}. By spacing the loops appropriately and constructing their topmost sections at slightly different heights, we ensure that any square covering points from two different variable loops covers either more than $4$ points or covers a crossing between those two loops. The configuration of the points at each crossing makes it so that every good square which contains points from two variable loops covers exactly two points from each of those loops. 

\begin{figure}[!htbp]
\centering
\includegraphics[scale=0.44]{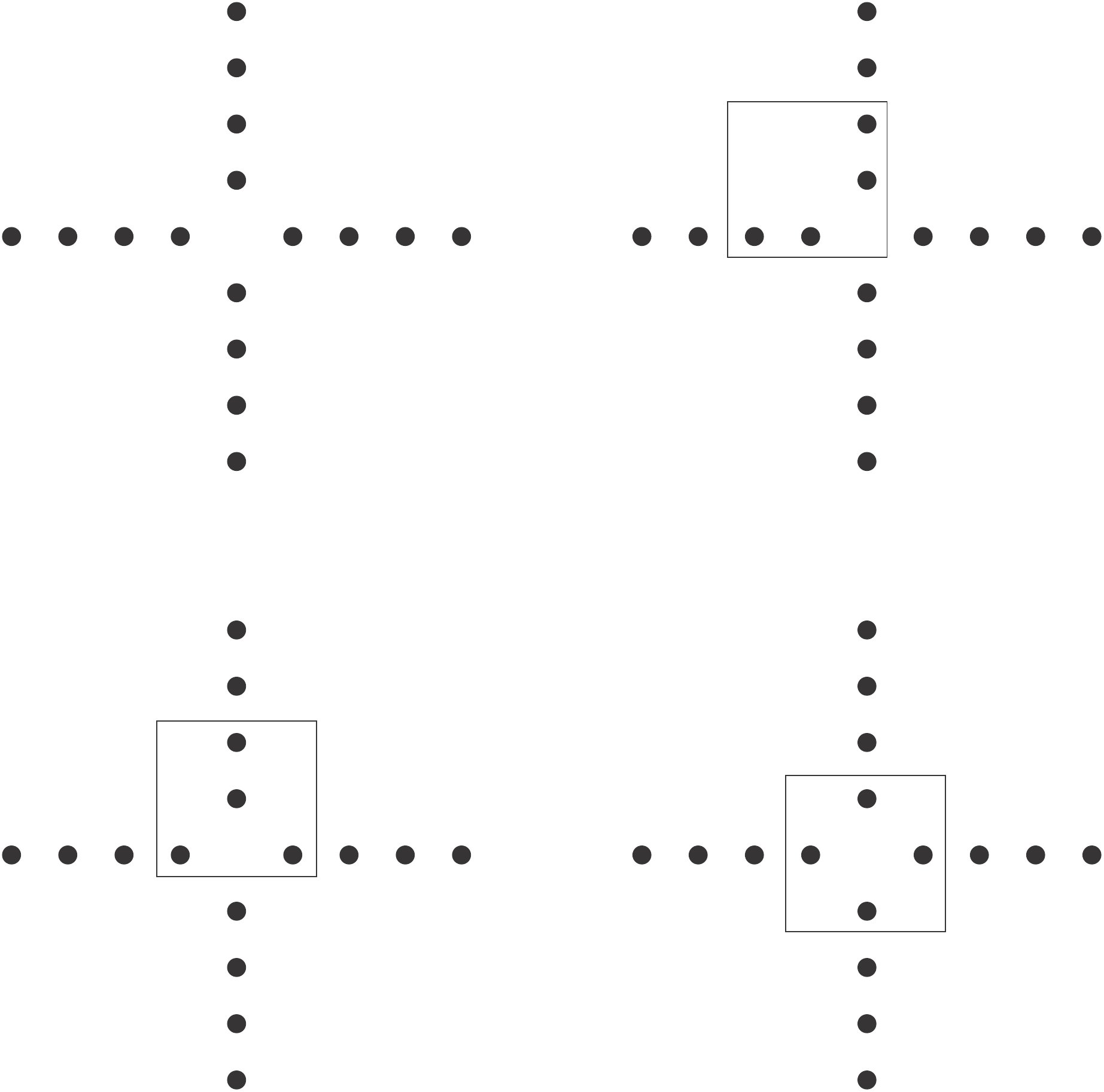}
\caption{Top left: placement of the points around a crossing between two variable loops. The other pictures depict all the essentially different ways in which a good square can cover the crossing.}
\label{fig:4}
\end{figure}

Figure \ref{fig:5} depicts the gadget used to simulate each clause. The configuration inside each of the $6$ red circles is designed so that any good square (inside the circle) which covers points from both the clause loop and the corresponding variable loop covers precisely two points from each. This way, any good square will cover an even number of points from each variable loop. The points of each variable loop are labeled (in order) from $1$ to $4t$ (for some $t$ that depends on the loop). We say that a good cover \textit{assigns} the value \textit{true} (resp. \textit{false}) to a variable if any two points labeled $2s$ and $2s+1$ (resp. $2s+1$ and $2s+2$) in the corresponding loop are contained in the same square, where the indices are taken modulo the total number of points in the loop. Clearly, a good cover assigns exactly one Boolean value to each variable. The points inside $c_{x,1}$ can be arranged so that if a good square that is contained in $c_{x,1}$ covers points from both the clause loop and the variable loop that corresponds to variable $x$, then it contains the points labeled with $4s$ and $4s+1$ if $x$ is not negated in the clause, or it contains the points labeled with $4s+1$ and $4s+2$ if $x$ appears in negated form ($\neg x$). Similarly, the points in $c_{x,2}$ are placed so that a good square which covers points from both the variable and the clause loops covers the points labeled as $4s+2$ and $4s+3$ if $x$ is not negated, or the points $4s+3$ and $4s+4$ if $x$ is negated. The points in $c_{y,1},c_{y,2},c_{z,1}$ and $c_{z,2}$ are arranged analogously. 

\begin{figure}[!htbp]
\centering
\includegraphics[scale=0.7]{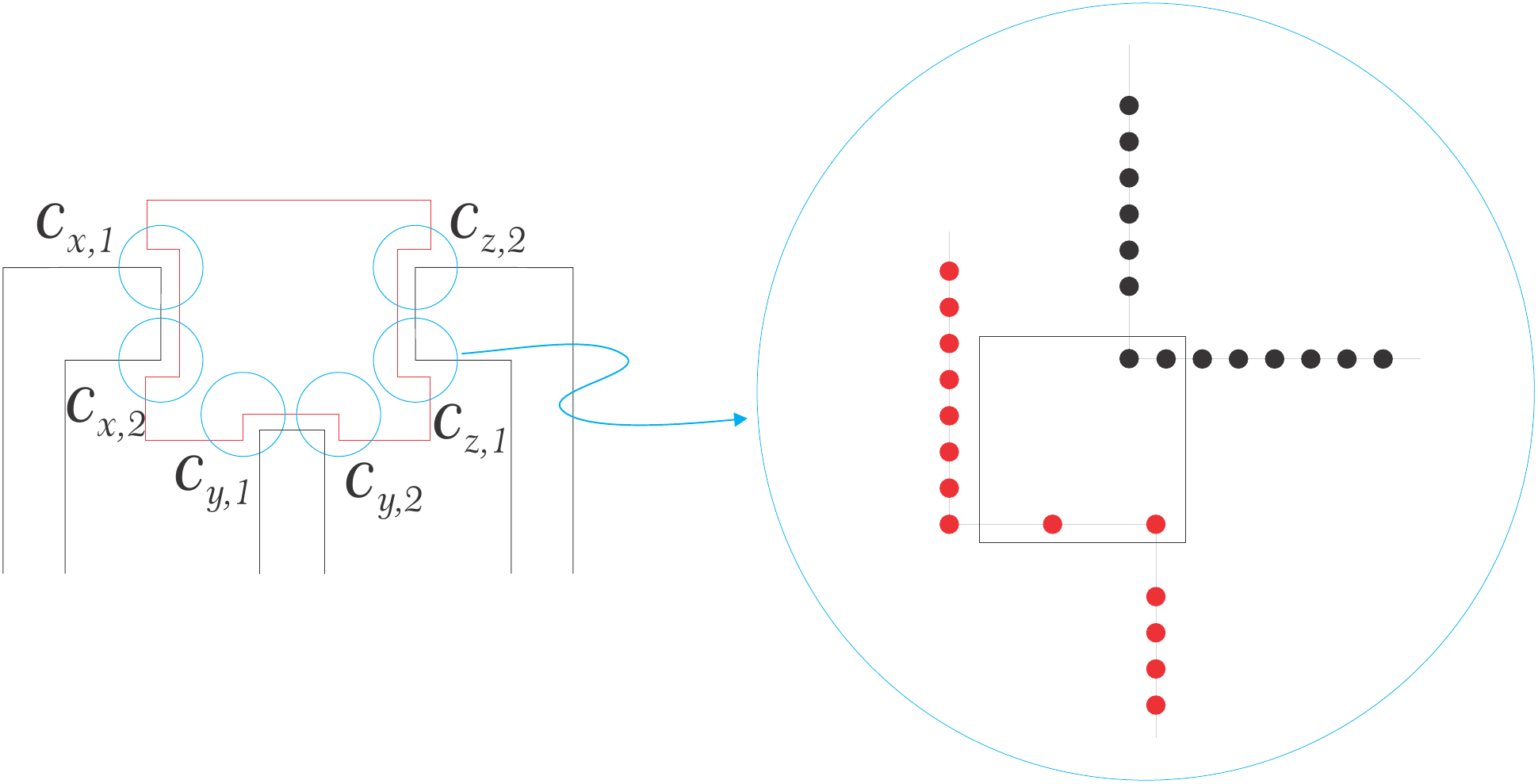}
\caption{Placement of the points around each clause loop. Up to reflection and rotation, the points inside each of the six blue circles are arranged as shown on the right. There is essentially a unique way of placing a good square that covers points from both the clause loop and the corresponding variable loop.}
\label{fig:5}
\end{figure}

Since the number of points of the clause loop is even but not a multiple of $4$, in any good cover there must be a square that contains two points from said loop and two points from one of the three corresponding variable loops. The construction described in the last paragraph makes it so that this is only possible if the cover assigns to one of the three variables the value that makes the clause true. Since this holds for all clause loops simultaneously, this shows that in order for a good cover to exist the formula must be satisfiable. We prove that the converse is true as well. Suppose that the formula is satisfiable and consider an assignment of Boolean values that satisfies it. For every clause choose a variable that has been assigned the correct value (with respect to the clause). Each variable loop can be covered by good squares which assign to it the correct value and such that one of these squares covers two points from each clause loop for which the variable was chosen (again, this is possible by the construction described above). The only thing that could go wrong when covering the variable loops is for the number of squares that cover two points from the variable loop to be odd, but this will not happen, since each variable corresponds to two loops and the number of crossings between any two variable loops is even. Since exactly two points from each variable loop have been covered, the number of points that still need to be covered in each variable loop is a multiple of four, so we can easily extend this collection of good squares to a good cover with ease. We have shown that the initial formula is satisfiable if and only if the point set admits a good cover.

It is not hard to see that the reduction can be carried out in a grid of polynomial dimensions and in polynomial time. This concludes the proof.
\end{proof}

\section{Matching points with homothets}\label{chap:matching}

\subsection{Toughness of {D}elaunay triangulations}\label{sec:tough}

\begin{theorem}\label{teo:tough}
Let $C\subset\mathbb{R}^2$ an $\alpha$-fat strictly convex body with smooth boundary and $S\subset\mathbb{R}^2$ a finite point set in $C$-general position such that no three points of $S$ lie on the same line. If $U\subset S$, then $D_C(S)-U$ has less than  \[\frac{450\degree-4\arcsin{\alpha}}{\arcsin{\alpha}}|U|+\frac{2\arcsin{\alpha}-90\degree}{\arcsin{\alpha}}\] connected components.

Of course, the result holds as long as $C$ can be made $\alpha$-fat by an affine transformation.
\end{theorem}

Note that as $\alpha$ goes to $1$ we get that $D_C(S)$ is $1$-tough, as was shown in \cite{simpletough} for Delaunay triangulations with respect to disks. We will need the following geometric lemma, which generalizes a well-known angular property of standard Delaunay triangulations. 

\begin{lemma}\label{teo:angles}
Let $C\subset\mathbb{R}^2$ an $\alpha$-fat convex body and $S\subset\mathbb{R}^2$ a finite point set. Suppose that $abc$ and $cda$ are two adjacent bounded faces of $D_C(S)$. We have that \[\measuredangle abc+\measuredangle cda\leq360\degree-2\arcsin{\alpha}.\]
\end{lemma}

\begin{proof}
The points $b$ and $d$ lie on different sides of the line that goes through $a$ and $c$. Also, 
since $(a,c)$ is an edge of $D_C(S)$, there is an homothet $C'$ of $C$ that contains $a$ and $c$ but contains neither $b$ nor $d$, we can actually choose $C'$ so that $a$ and $c$ lie on its boundary. This is all the information that we need in order to deduce the result.

By translating and rescaling, we may assume that $\alpha B^2\subset C'\subset B^2$. The points $a$ and $c$ are not contained in $\alpha B^2$, since they lie on the boundary of $C$. The fact that $C$ is convex implies that the convex hull $\text{conv}(\alpha B^2\cup\{a,c\})$ does not contain $b$ and $d$ (see figure \ref{fig:6} a). It is possible to slide $b$ and $d$ until they lie on the boundary of $\text{conv}(\alpha B^2\cup\{a,c\})$ without decreasing the values of $\measuredangle abc$ and  $\measuredangle cda$, so we may and will assume that they lie on said boundary. By a similar argument, it suffices to prove the inequality under the assumption that $a$ and $c$ lie on the boundary of $B$ (see figure \ref{fig:6} b).

\begin{figure}[!htbp]
\centering
\includegraphics[scale=0.6]{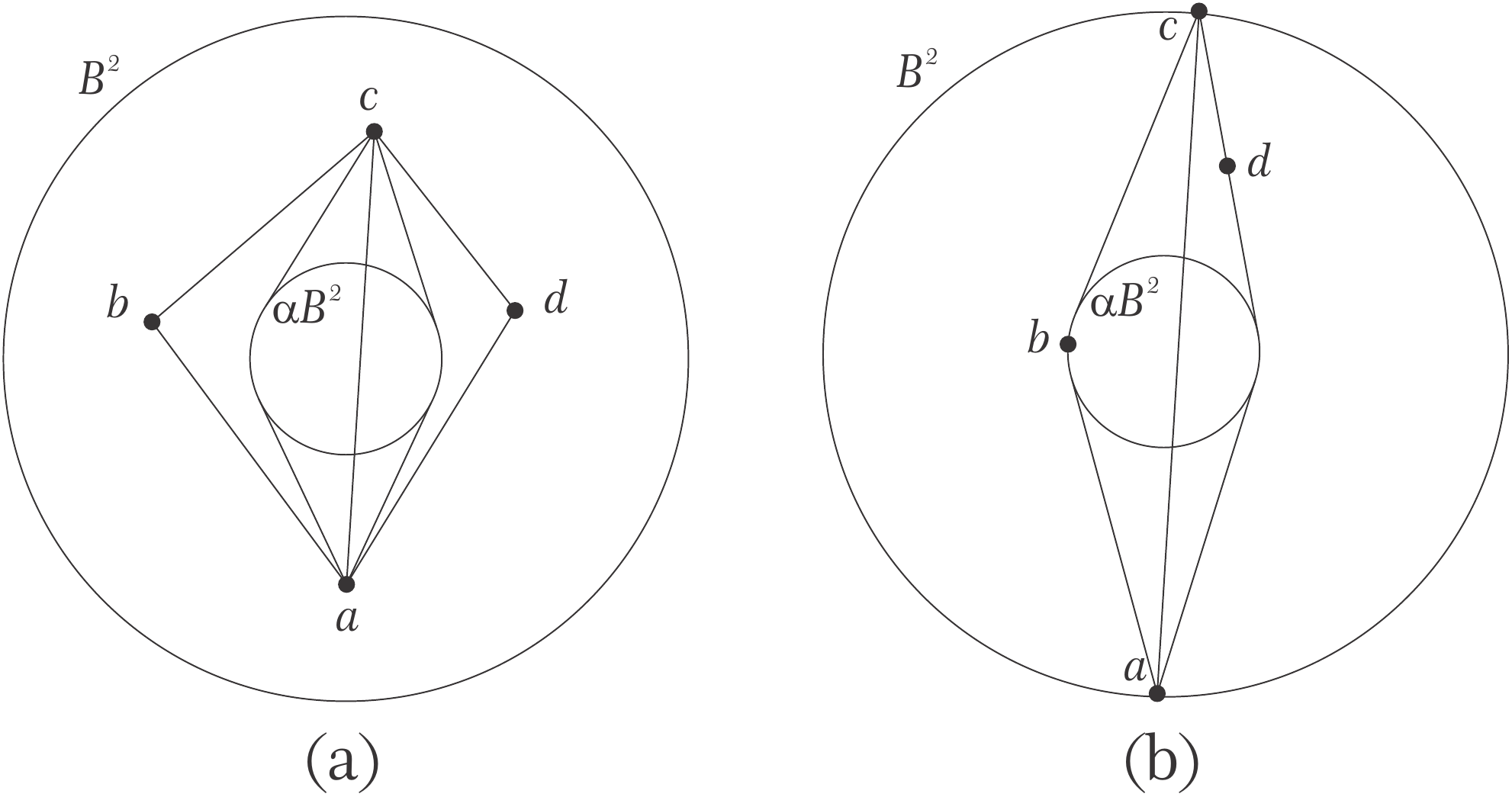}
\caption{Configuration in the proof of Lemma \ref{teo:angles}}
\label{fig:6}
\end{figure}

It is not hard to see that $\measuredangle abc$ grows larger as $b$ gets closer to either $a$ or $c$. Similarly, $\measuredangle cda$ grows larger as $d$ gets closer to either $a$ or $c$.  Thus, $\measuredangle abc+\measuredangle cda\leq 360\degree-\Theta$, where $\Theta$ is the measure of the angle at $a$ (or, equivalently, $c$) of $\text{conv}(\alpha B^2\cup\{a,c\})$. A simple calculation shows that $\Theta\geq 2\arcsin\alpha$, with equality if an only if the segment joining $a$ to $c$ goes through the closure of $\alpha B$.
\end{proof}

Instead of trying to prove Theorem~\ref{teo:tough} directly, we first bound the size of an independent set\footnote{A set of vertices of a graph forms an \textit{independent set} if no two of them are adjacent.} in $D_C(S)$.

We return to the proof of 

\begin{theorem}\label{teo:independent}
Let $C$ and $S$ be as in the statement of Theorem~\ref{teo:tough} and $I\subset S$ an independent set of vertices of $D_C(S)$. Then \[|I|<\frac{450\degree-4\arcsin{\alpha}}{450\degree-3\arcsin{\alpha}}|S|+\frac{90\degree-2\arcsin{\alpha}}{450\degree-3\arcsin{\alpha}}.\]
\end{theorem}

\begin{proof}
Let $S'=S\backslash I$ and notice that at least one vertex $u$ of the outer face of $D_C(S)$ must belong to $S'$. For each edge of $D_C(S)$ consider an homothet of $C$ that contains its endpoints and no other element of $S$, and take two points $v,w\notin S$ which are not contained in any of those circles and such that the triangle with vertices $u,v$ and $w$ contains all points of $S$. By the choice of $v$ and $w$, the Delaunay triangulation $D_C(S\cup\{v,w\})$ contains $D_C(S)$ as a subgraph (see figure \ref{fig:7}). Let $D'$ the subgraph of $D_C(S\cup\{v,w\})$ induced by $S'\cup\{v,w\}$. Since $I$ is an independent set of $D_C(S\cup\{v,w\})$ and contains no vertex of the outer face, each point in $I$ corresponds to a bounded face of $D'$ which is bounded by a cycle and is not a face of $D_C(S\cup\{v,w\})$. The previous observation shows, in particular, that $D'$ is connected. Following the terminology in~\cite{simpletough}, we classify the bounded faces of $D'$ as \textit{good faces} if they are also faces of $D_C(S\cup\{v,w\})$, and as \textit{bad faces} if they contain one point of $I$; note that each bounded face falls in exactly one of these two categories. Let $g$ and $b=|I|$ be the number of good and bad faces, respectively.

We will asign some \textit{distinguished angles} to each edge of $D'$. If $(p,q)$ is an interior edge of $D'$ then it is incident to two bounded faces $pqr$ and $qps$ of $D_C(S\cup\{v,w\})$; we assign the edge $(p,q)$ to the angles $\angle qrp$ and $\angle psq$. Each exterior edge $(p,q)$ is incident to a single such face $pqr$; we assign $(p,q)$ to $\angle qrp$ (see figure \ref{fig:7}). On one hand, all three angles of any good face are distinguished and add up to $180\degree$. On the other hand, every bad face contains a point of $I$ and all angles of $D_C(S\cup\{v,w\})$ which are anchored at that point are distinguished and add up to $360\degree$. The total measure of the distinguished angles is thus \[T=g\cdot180\degree+b\cdot360\degree.\]

This quantity can also be bounded using Lemma~\ref{teo:angles}, as follows. Each edge of $D'$ is assigned to at most two distinguished angles, which have total measure at most $360\degree-2\arcsin{\alpha}$ (indeed, this is trivial if there is only one such angle, and it follows from the lemma if there are two). By Euler's formula, the number of edges of $D'$ is $|S'\cup\{v,w\}|+(b+g+1)-2=|S|+g+1$. Each of the three edges on the outer face is assigned to only one angle, so summing over all edges we get \[T<(360\degree-2\arcsin{\alpha})(|S|+g-2)+3\cdot 180\degree,\] whence \[g\cdot180\degree+b\cdot360\degree<(360\degree-2\arcsin{\alpha})(|S|+g-2)+540\degree.\] Since each element of $I$ is incident to at least three faces of the triangulation $D_C(S\cup\{v,w\})$ we get, again by Euler's formula, that \[3(|S|+2)-6\geq g+3b,\] so $g\leq 3(|S|-b)$. We momentarily set $\beta=2\arcsin{\alpha}$, then the two inequalities yield \[b\cdot 360\degree<(360\degree-\beta)|S|+(180\degree-\beta)(3|S|-3b)-2(360\degree-\beta)+540\degree,\] \[(900\degree-3\beta)b<(900\degree-4\beta)|S|-(180\degree-2\beta),\] \[|I|= b<\frac{900\degree-4\beta}{900\degree-3\beta}|S|-\frac{180\degree-2\beta}{900\degree-3\beta},\] and the result follows.
\end{proof}

\begin{figure}[!htbp]
\centering
\includegraphics[scale=0.85]{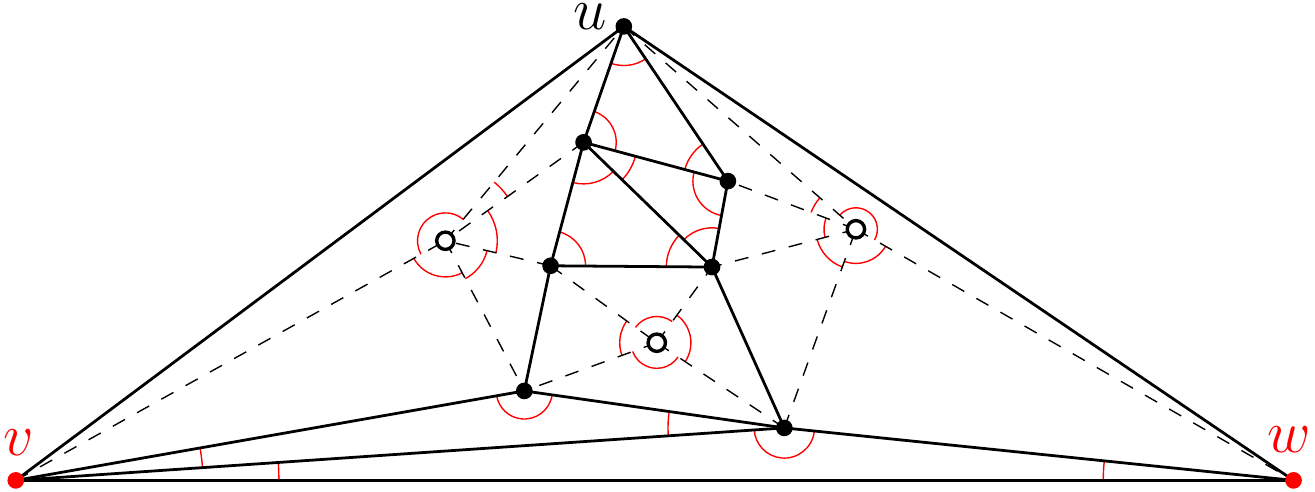}
\caption{An example of how the Delaunay triangulation $D_C(S\cup\{u,w\})$} might look. All distinguished angles are marked in red. \textbf{This figure, which appeared in \cite{simpletough}, was provided to us by Ahmad Biniaz.}
\label{fig:7}
\end{figure}

The following simple lemma extends a result used in~\cite{simpletough}.

\begin{lemma}\label{teo:innerpath}
Let $C\subset\mathbb{R}^2$ a strictly convex body and $S\subset\mathbb{R}^2$ a finite point set in $C$-general position. Consider an homothet $C'$ of $C$ whose boundary contains exactly two points, $p$ and $q$ say, of $S$. Then $p$ and $q$ are connected by a path in $D_C(S)$ that lies in $C'$.
\end{lemma}

\begin{proof}
The proof is by induction on the number of points $t$ contained in the interior of $C'$. If $t=0$, then $p,q$ are adjacent in $D_C(S)$ and we are done. Otherwise, let $r$ be a point in the interior of $C'$ and apply a dilation with center $p$ until the image of $C'$ has $r$ on its boundary, we call this homothet $C_1$, repeat this process but now with center $q$ and call the resulting homothet $C_2$. This way, $p$ and $r$ lie on the boundary of $C_1$, while $q$ and $r$ lie on the boundary of $C_2$; notice also that $C_1,C_2\subset  C'$. Since $C$ is strictly convex, we can ensure that the boundaries of $C_1$ and $C_2$ contain no point of $S$ other than $p,r$ and $q,r$, respectively, by taking a small perturbation of the homothets if necessary. Notice that the interiors of each of $C_1,C_2$ contain at most $t-1$ points of $S$. Thus, by the inductive hypothesis, we can find two paths joining $p$ to $r$ and $q$ to $r$ inside $C_1$ and $C_2$, respectively. The union of the two paths we just mentioned contains a path from $p$ to $q$ that lies completely in $C'$, as desired. See figure \ref{fig:8}.
\end{proof}

\begin{figure}[!htbp]
\centering
\includegraphics[scale=0.44]{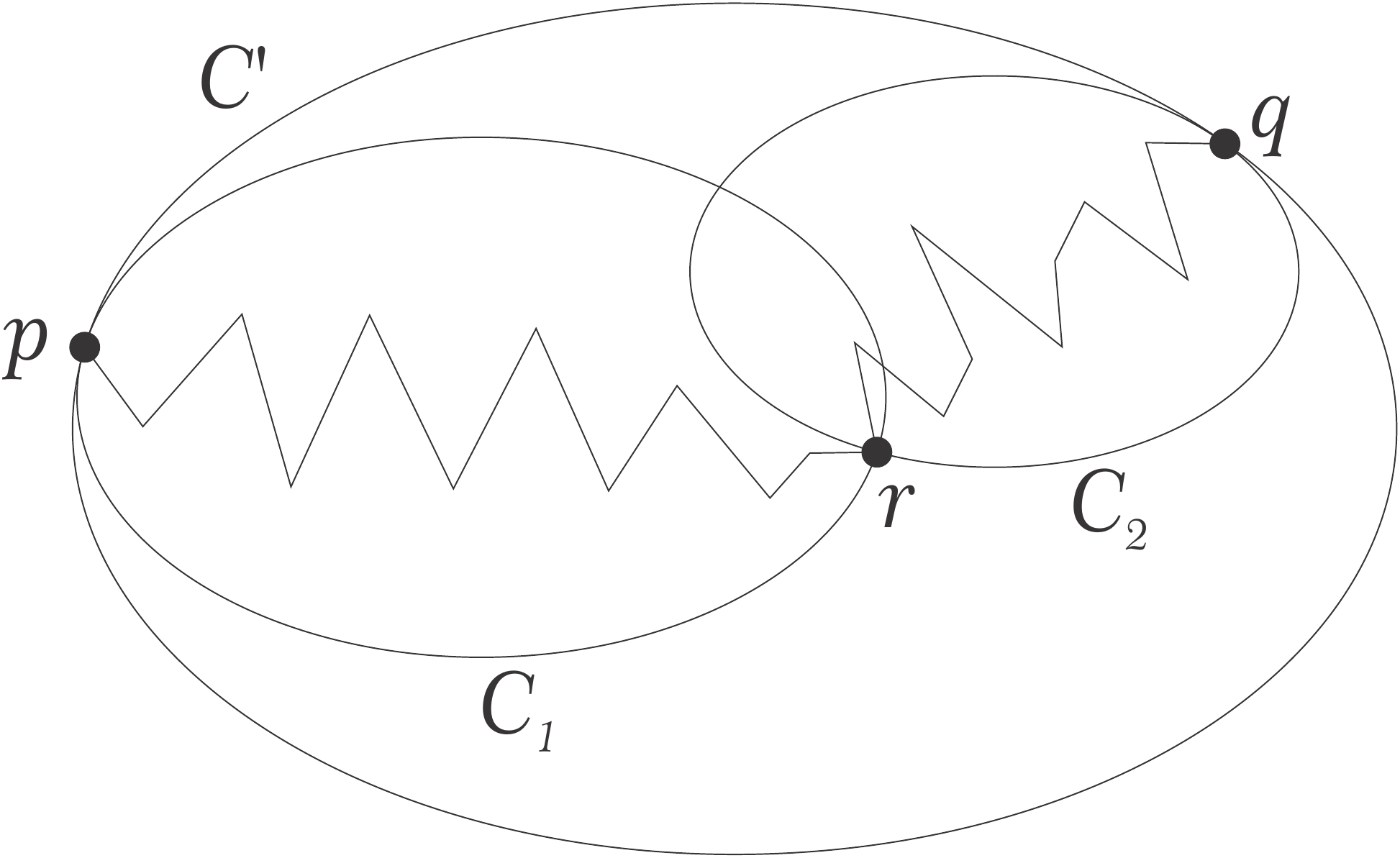}
\caption{Configuration in the proof of Lemma \ref{teo:innerpath}.}
\label{fig:8}
\end{figure}

Theorem~\ref{teo:tough} is an easy consequence of Theorem~\ref{teo:independent} and Lemma~\ref{teo:innerpath}. Indeed, consider an arbitrary set of vertices $U\subset S$ and choose a representative vertex from each component of $D_C(S)-U$. Let $V$ be the set of all representative vertices and consider the Delaunay triangulation $D_C(U\cup V)$. Suppose that there is an edge in this graph between two vertices $p$ and $q$ of $V$, then there is an homothet $C'$ such that $C'\cap (U\cup V)=\{p,q\}$. Furthermore, by applying a slight perturbation if necessary, we may assume that $C'$ contains no other point of $S$ on its boundary. Lemma~\ref{teo:innerpath} now tells us that there is a path in $D_C(S)$ joining $p$ and $q$ which lies in $C'$. Since $p$ and $q$ lie in different components of $D_C(S)-U$, this path must contain at least one vertex from $U$, which must therefore lie in $C'$. This contradiction shows that $V$ is an independent set of $D_C(U\cup V)$.  By Lemma~\ref{teo:independent}, \[|V|<\frac{450\degree-4\arcsin{\alpha}}{450\degree-3\arcsin{\alpha}}|(V\cup U)|-\frac{90\degree-2\arcsin{\alpha}}{450\degree-3\arcsin{\alpha}},\] \[|V|<\frac{450\degree-4\arcsin{\alpha}}{\arcsin{\alpha}}|U|-\frac{90\degree-2\arcsin{\alpha}}{\arcsin{\alpha}},\] but $|V|$ is just the number of components of $D_C(S)-U$, so we are done.

\subsection[Large matchings in Delaunay graphs]{Large matchings in $D_C(S)$}\label{sec:matchings}
For any graph $G$, let $o(G)$ denote the number of connected components of $G$ which have an odd number of vertices. The Tutte-Berge formula~\cite{tutte-berge} tells us that the size of the maximum matching in a graph $G$ with vertex set $V$ equals \[\frac{1}{2}\left(|V|-\max_{U\subset V}\{o(G-U)-|U|\}\right).\]

Combining Theorem~\ref{teo:tough} and the Tutte-Berge formula yields the main result of this sections.

\begin{theorem}\label{teo:matching}
Let $C\subset\mathbb{R}^2$ an $\alpha$-fat strictly convex body with smooth boundary and $S\subset\mathbb{R}^2$ a finite point set in $C$-general position such that no three points of $S$ lie on the same line. Then $D_C(S)$ contains a matching of size at least \[\left(\frac{1}{2}-\frac{450\degree-5\arcsin{\alpha}}{900\degree-6\arcsin{\alpha}}\right)|S|+\frac{45\degree-\arcsin{\alpha}}{450\degree-4\arcsin{\alpha}}\left(1+\frac{450\degree-5\arcsin{\alpha}}{450\degree-3\arcsin{\alpha}}\right).\]
Again, the result also holds if $C$ can be made $\alpha$-fat by an affine transformation.
\end{theorem}

\begin{proof}
Let $U\subset S$ and notice that $o(D_C(S)-U)$ is at most the number of connected components of $D_C(S)-U$. Whence, Theorem \ref{teo:tough} implies that \[o(D_C(S)-U)<\frac{450\degree-4\arcsin{\alpha}}{\arcsin{\alpha}}|U|-\frac{90\degree-2\arcsin{\alpha}}{\arcsin{\alpha}}.\] Together with $o(D_C(S)-U)+|U|\leq S$, this can be seen to imply that  $|S|-(o(D_C(S)-U)-|U|)$ must be larger than \[\left(1-\frac{450\degree-5\arcsin{\alpha}}{450\degree-3\arcsin{\alpha}}\right)|S|+\frac{90\degree-2\arcsin{\alpha}}{450\degree-4\arcsin{\alpha}}\left(\frac{450\degree-5\arcsin{\alpha}}{450\degree-3\arcsin{\alpha}}+1\right),\] and the result follows.
\end{proof}

To conclude this sections, we obtain a weaker bound that holds under more general conditions.

\begin{theorem}\label{teo:matching2}
Let $C\subset\mathbb{R}^2$ be a strictly convex body. Then, for every finite set $S\subset\mathbb{R}^2$ we have that $f(C,2,S)\leq|S|-\lceil\frac{1}{3}(|S|-8)\rceil$.
\end{theorem}

\begin{proof}
We will essentially show that $D_C(S)$ (which is planar, but not necessarily a triangulations) can be turned into a planar graph of minimum degree at least three by adding a constant number of vertices, the theorem then follows from a result of Nishizeki and Baybars~\cite{planarmatchings}. 

For every $x$ (not necessarily in $S$) on the boundary of $C$, let $A_x$ be the smallest closed angular region which has $x$ as its vertex and contains $C$, and $\alpha_x\leq 180\degree$ be the measure of the angle that defines $A_x$. Let $a_x=(A_x-x)\cap\mathbb{S}^2$, $a_x$ is an arc of $\mathbb{S}^2$ determined by an angle of measure $\alpha_x$. See figure \ref{fig:9}.

\begin{figure}[!htbp]
\centering
\includegraphics[scale=0.52]{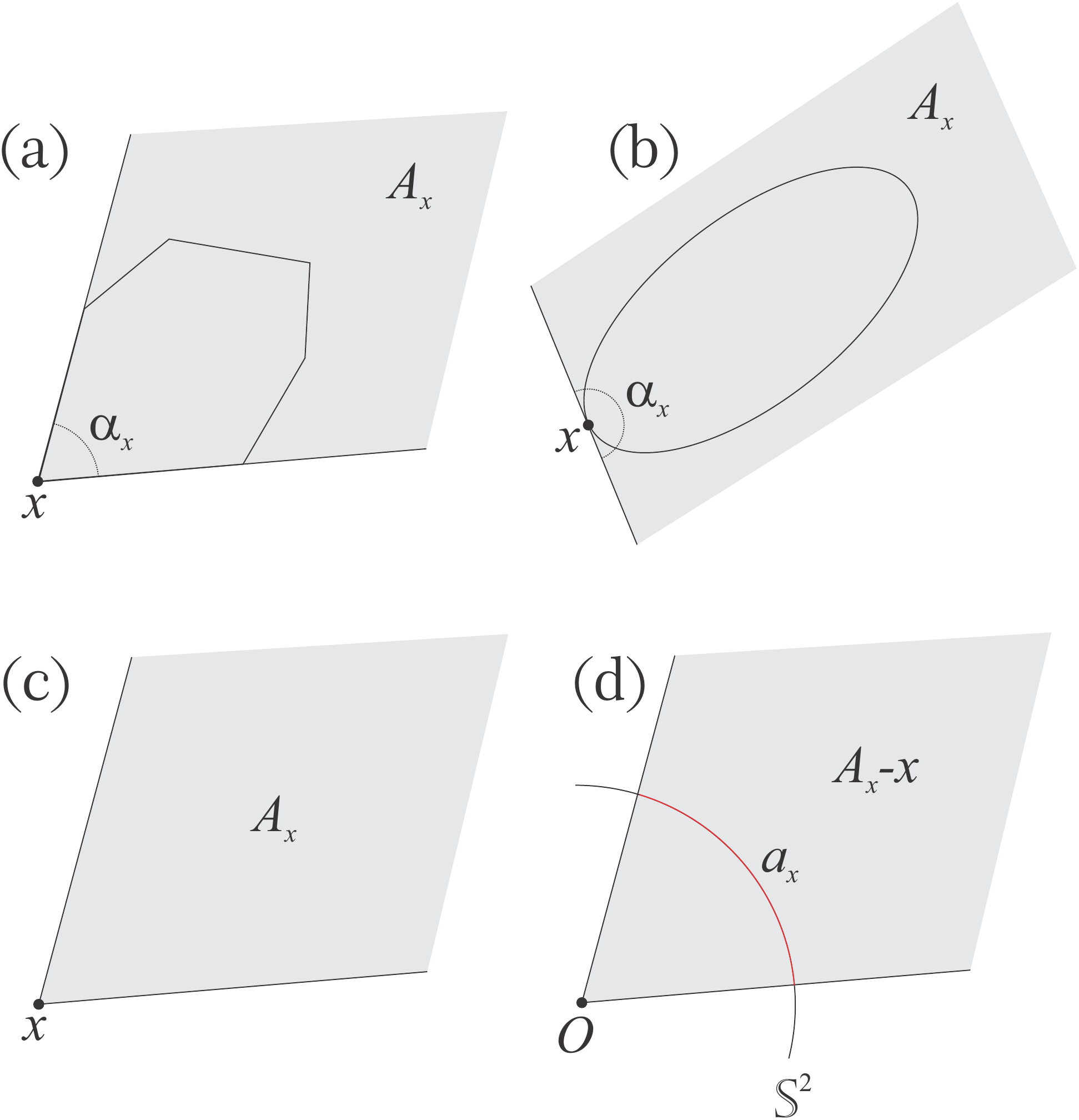}
\caption{(a),(b): two examples of $A_x$ and $\alpha_x$. (c),(d): how $A_x$, $A_x-x$ and $a_x$ might look.}
\label{fig:9}
\end{figure}

\begin{lemma}\label{teo:fivepoints}
There are five points in $\mathbb{S}^2$ such that, for every $x$ on the boundary of $C$, $a_x$ contains at least one of these points in its interior.
\end{lemma}

\begin{proof}
Let $x_1,x_2,...,x_r$ be distinct points on the boundary of $C$. The intersection $\cap^{r}_{i=1}A_{x_i}$ is a closed and convex polygonal region and a quick calculation shows that $\sum^{r}_{i=1}\alpha_{x_i}\geq (r-2)180\degree$, where the equality occurs if and only if $C$ is an $r$-agon with vertices $x_1,...,x_r$. Since the result is easily seen to be true if $C$ is either a triangle or a quadrilateral, we can assume that, for any distinct points $x,y,z$ and $w$ on the boundary of $C$, $\alpha_{x}+\alpha_{y}+\alpha_{z}>180\degree$ and $\alpha_{x}+\alpha_{y}+\alpha_{z}+\alpha_{w}>360\degree$. 

Let $A_{90\degree}$ be the set that consists of all points $x$ on the boundary of $C$ such that $\alpha_x\leq 90\degree$, then $|A_{90\degree}|\leq 3$. If $|A_{90\degree}|\leq 2$, we take four points in $\mathbb{S}^2$ such that they are the vertices of a square and that one of them is contained in the arc $a_x$ determined by one of the elements of $A_{90\degree}$. This set of four points hits the interiors of all but at most one of the arcs $a_x$, so it is possible to find five points in $\mathbb{S}^2$ which hit the interiors of all arcs. If $|A_{90\degree}|=3$, then $\sum_{x\in A_{90\degree}}\alpha_x>180\degree$. By choosing a square $Q$ with vertices in $\mathbb{S}^2$ uniformly at random, with positive probability $Q$ will be such that $v$ is contained in the interior of $a_x$ for more than two pairs $(v,x)$ where $v$ is a vertex of $Q$ and $x\in A_{90\degree}$. Since no arc $a_x$ with $x\in A_{90\degree}$ may contain more than one vertex of $Q$, every arc appears in at most one of the pairs. This implies that, with positive probability, the vertices of $Q$ hit the interior of every arc $a_x$ for $x\in A_{90\degree}$, but they clearly also hit the interior of every other $a_x$ and, thus, there is a set of four points (to which we can add any other point of $\mathbb{S}^2$ so that it has five elements) with the desired property.
\end{proof}

Let $x_1,x_2,x_3,x_4,x_5$ be five points as in Lemma~\ref{teo:fivepoints}. Consider a very large positive real number $\gamma$ to be specified later and let $S'=S\cup\{\gamma x_1,\gamma x_2,\dots,\gamma x_5\}$.

\begin{claim}
If $\gamma$ is large enough then every point of $S$ has degree at least $3$ in $D_C(S')$. 
\end{claim}

\begin{proof}
Let $s\in S$ and consider an arbitrary line $\ell$ with $\ell\cap S=\lbrace s\rbrace$ and an open halfplane $H$ determined by $\ell$, we show that if $\gamma$ is large enough then $s$ is adjacent to a point in $H$. Assume, w.l.o.g, that $\ell$ is vertical and that $H$ is the right half-plane determined by $\ell$ and let $x_H$ be the leftmost point of $C$. Observe that, by Lemma~\ref{teo:fivepoints}, for any large enough $\gamma$ the angular region $A_{x_H}-x_H+s$ contains at least one of the points $\gamma x_1,\gamma x_2,\dots,\gamma x_5$. Now, consider the smallest $\lambda>0$ such that the homothet $C_\lambda=\lambda(C-x_H)+s$ contains at least two points of $S'$ (it exists, since $C_\lambda$ will contain $s$ and at least one of $\gamma x_1,\gamma x_2,\dots,\gamma x_5$ if $\lambda$ is very large). If necessary, perturb $C'$ slightly so that it contains $s$ and exactly one other element of $S'$, then this element lies in $H$ and is adjacent to $s$, as desired. This implies that, for large enough $\gamma$, the neighbours of $s$ are not contained in a closed halfplane determined by a line through $s$, which is only possible if $s$ has degree at least $3$ in $D_C(S')$. Any large enough $\gamma$ will ensure that this holds simultaneously for every $s\in S$.
\end{proof}

The result clearly holds for $|S|\leq 8$, so we assume that $|S|>8$. Let $X\subset\{\gamma x_1,\gamma x_2,\dots,\gamma x_5\}$ be the set of $\gamma x_i$'s which are adjacent to at least one point of $S$ and delete the rest of the $\gamma x_i$'s from $D_C(S')$. It is not hard to see that $|X|\geq 2$. If $|X|=2$, join these two points of by an edge (skip this step if they are already adjacent) and add a vertex $v$ in the outer face of $D_C(S')$, then connect $v$ to both element of $X$ and to some point in $S$ while keeping the graph planar. Otherwise, if $|X|\geq 2$, we can add edges between the elements of $X$ so that there is a cycle of length $|X|$ going through all of them and the graph remains planar. In any case, the resulting graph is simple, planar, connected, and it has at least $|S|+3>10$ vertices, all of degree at least three. Nishizeki and Baybars \cite{planarmatchings} showed that any graph with these properties contains a matching of size at least $\lceil\frac{1}{3}(n+2)\rceil$, where $n$ is the total number of vertices. Let $t\leq 5$ denote the number of vertices that do not belong to $S$. Deleting all vertices not in $S$ from the graph, we get a matching in $D_C(S)$ of size at least $\lceil\frac{1}{3}(|S|+2+t)\rceil-t=\lceil\frac{1}{3}(|S|-8)\rceil$. This matching translates into a way of covering $S$ using no more than $|S|-\lceil\frac{1}{3}(|S|-8)\rceil$ $2^+/S$-homothets of $C$.
\end{proof}

\section{Further research and concluding remarks}\label{chap:outro}

\subsection*{A drawback of the lower and upper densities}

Unlike the standard upper and lower densities of an arrangement, the measure theoretic versions introduced in Section \ref{sec:packcover} are in general not independent of the choice of the origin. The reason for this is that, for any two points $O_1$ and $O_2$, the measures of the balls $B(O_1,r)$ and $B(O_2,r)$ may differ in an arbitrarily large multiplicative constant for every $r$. Although this can be avoided by adding the requirement that $\mu(X)\leq c\cdot \text{vol}(X)$ for any compact $X$ and some constant $c$, this defect begs the question: Is there a better way of extending the standard definitions to arbitrary Borel measures?

\subsection*{Bounds in the other direction}

The hidden constants $c_{f,d}$ and $ c_{g,d}$ obtained in the proofs of theorems \ref{teo:fbound} and \ref{teo:gbound} increase and decrease exponentially in $d$, respectively. We showed in Sections \ref{sec:coverdensity} and \ref{sec:packdensity} that, under the right conditions, $\Theta_H(\mu,C)\leq c_{f,d}$ and $\delta_H(\mu,C)\geq c_{g,d}$ (in the case of measures). This yields, in particular, that $c_{f,d}\geq\Theta_H(C)$ and $c_{g,d}\leq\delta_H(C)$ for any $C$ (we remark that this can also be obtained by considering the restriction of the Lebesgue measure to large boxes). Both of these bounds also hold for the hidden constants in the case of point sets, as can be shown by taking a sufficiently large section of a grid.

\begin{theorem}\label{teo:grid}
Let $C\subset\mathbb{R}^d$ be a convex body and $\epsilon$ any positive real number. Then, for any sufficiently large $k$, there is an integer $N(C,\epsilon,k)$ such that for each $N$ with $N>N(C,\epsilon,k)$ the set $[N]^d=\lbrace(x_{1},x_{2}, ...,x_{d})\in\mathbb{R}^d\text{ }|\text{ }x_{i}\in [N]\rbrace$\footnote{For each positive integer $n$, $[n]$ denotes the set $\{1,2,\dots,n\}$.} of integer points inside a $d$-hypercube of side $N$ satisfies $f(C,k,[N]^d)>(\Theta_H(C)-\epsilon)\frac{N^{d}}{k}$ and $g(C,k,[N]^d)<(\delta_H(C)+\epsilon)\frac{N^{d}}{k}$.
\end{theorem}

Since the proof is quite straightforward, we give only a sketch of the bound for $f$.

\begin{proof}
Let $\delta>0$. For any sufficiently large $k$, there is an homothet $C'$ of $C$ of volume less than $(1+\delta)k$ which has the following property: Every homothet $C_{1}$ of $C$ that covers at most $k$ points of the lattice $\mathbb{Z}^{d}$ is contained in a translate $C_{2}$ of $C'$ such that every point covered by $C_{1}$ has distance at least $\sqrt{d}$ from the boundary of $C_{2}$. Also, for any sufficiently large $N$, the set $[1,N]^d=\lbrace(x_{1},x_{2}, ...,x_{d})\text{ }|\text{ }1\leq x_{i}\leq N\rbrace$ cannot be covered by less than $(\Theta_H(C)-\delta)\frac{N^{d}}{(1+\delta)k}$ translates of $C'$. 

Now, consider a cover of $[N]^d$ by $k^-/[N]^d$-homothets of $C$ and for each of these homothets take a translate of $C'$ with the described properties. This way, we get a cover of $[1,N]^d$ with translates of $C'$, and the result follows by taking a small enough $\delta$. This is not entirely correct, since the $k^d/[N]^d$-homothets which are not completely contained in $[1,N]^d$ may not fit inside a translate of $C'$ in the desired way, but these become insignificant if we choose $N(C,\epsilon,k)$ to be large enough.
\end{proof}

While this shows that the exponential growth of $c_{f,d}$ and exponential decay of $ c_{g,d}$ are necessary, we believe that these bounds are still far from optimal. It might be an interesting problem to try and find point sets or measures for which $f$ is large (or $g$ is small) with respect to $\frac{|S|}{k}$ (or $\frac{\mu\mathbb{R}^d}{\mu(C)}$). 

\begin{prob}
What are the optimal values of $c_{f,d}$ and $c_{g,d}$? 
\end{prob}

Given that determination of packing and covering densities tends to be a very difficult problem, one should expect an exact solution to the problem above to be out of reach (for now). Similar questions can be asked for the results in Section \ref{sec:matchings}.

\begin{prob}
Can theorems \ref{teo:matching} and \ref{teo:matching2} be improved?
\end{prob}

\subsection*{Higher order Voronoi diagrams}

In their point set versions, theorems \ref{teo:fbound} and \ref{teo:gbound} can be interpreted as a kind of structural property of the order-$k$ Voronoi diagram of $S$ with respect to the (not necessarily symmetric) distance function induced by $C$. The cells in this diagram encode the $k$-element subsets of $S$ that can be covered by an homothet of $C$ which contains exactly $k$ points of $S$. See \cite{aurenhammervoronoi} for more on Voronoi diagrams.

\subsection*{Beyond convex bodies}

While the assumptions that $C$ is bounded and has nonempty interior can both easily be seen to be essential to the results obtained in sections \ref{chap:cover} and \ref{chap:pack}, the convexity hypothesis can be somewhat relaxed:

The \textit{kernel} of a compact connected set $C\subset\mathbb{R}^d$, denoted by $\text{ker}(C)$, is the set of points $p\in C$ such that for every other $q\in C$ the segment with endpoints $p$ and $q$ is completely contained in $C$. We say that $C$ is $\textit{star-shaped}$ if $\text{ker}(C)\neq\emptyset$. Our results in sections \ref{chap:cover} and \ref{chap:pack} remain true as long as $C$ is star-shaped and there is an affine transformation $T$ such that $B^d\subset ker(C)\subset C\subset\alpha B^d$ for some $\alpha=\alpha(d)$ that depends only on $d$.

A sufficiently large grid (as in Theorem~\ref{teo:grid}) or the restriction of the Lebesgue measure to a large box show that we can not hope to extend theorems ~\ref{teo:fbound} and ~\ref{teo:gbound} to non-convex bodies while keeping the hidden constant independent of $C$.

\subsection*{Complexity}
Even though the reduction to $3$-SAT given in Section \ref{sec:complexity} and the proof of NP-hardness in \cite{matchingnp} work only in some very particular cases, we conjecture the following.

\begin{conj}
Let $C$ be a convex body and $k\geq 3$ an integer, then $C$-$k$-COVER is NP-hard. Similarly $C$-$k$-PACK is NP-hard for every $k\geq 2$.
\end{conj}

\subsection*{Covering with disjoint homothets}
It is natural to ask whether a  result along the lines of Theorem~\ref{teo:fbound} holds if we require that the $k^-/S$-homothets in the cover have disjoint interiors. A sufficiently fine grid (in the case of point sets) and the restriction of Lebesgue measure to a bounded box (in the measure case) show that, in general, this is not the case, indeed, unless  $\theta(C)=1$, the number of interior-disjoint $k^-/S$-homothets required in these cases will not be bounded from above by a function of $\frac{|S|}{k}$ ($\frac{\mu(\mathbb{R}^d)}{\mu(C)}$, respectively). Perhaps the most annoying unanswered questions are the following.

\begin{prob}
Let $S$ be a finite set of at least $k$ points in the plane and $C$ a square. Is the number of disjoint homothets required to cover $S$ bounded from above by a function of $\frac{|S|}{k}$? Is it $O(\frac{|S|}{k})$? What is the answer if we add the restriction that no two points of $S$ lie on the same horizontal or vertical line?
\end{prob}

We believe the answer to all the previous questions to be no. In fact, we suspect that a family of examples which exhibit this can be constructed along the following lines:

Set $k$ to be very large and start by taking a uniformly distributed set of about $k$ points inside the unit square. Choose $m$ points (with $m$ much smaller than $k$) inside the square such that the set of their $2m$ $x$ and $y$ coordinates is independent over $\mathbb{Q}$ and place $k$ points around a very small neighborhood of each of these $m$ points. It is not hard to see that this would work directly (even for $m=1$) if all the squares in the cover were required to lie inside the unit square. This example can be adapted to measures as well.

For $k=2$, this problem is equivalent to the study of strong matchings; see Section \ref{sec:previouswork} for details.

\subsection*{Weak nets for zonotopes}

A centrally symmetric convex polytope is a \textit{zonotope} if all its faces are centrally symmetric\footnote{A zonotope is commonly defined as the set of all points which are linear combinations with coefficients in $[0,1]$ of a finite set of vectors, but the alternative definition given here, which is widely known to be equivalent, serves our purpose much better.}. Notice that each face of a zonotope is a zonotope itself. Examples of zonotopes include hypercubes, parallelepipeds and centrally symmetric convex polygons.

For zonotopes with few vertices, the following geometric lemma can act as a substitute of \ref{teo:hit}, allowing us to construct even smaller weak $\epsilon$-nets. 

\begin{lemma}\label{teo:zonotopevertices}
Let $Z\subset\mathbb{R}^d$ be a zonotope and consider two homothets $Z_1$ and $Z_2$ of $Z$ with non-empty intersection. If $Z_1$ is at least as large as $Z_2$, then it contains at least one vertex of $Z_2$.
\end{lemma}

\begin{proof}
We proceed by induction on $d$. The result is trivial for $d=1$ (here, $Z\subset\mathbb{R}$ is simply an interval). Let $p_{1}$ and $p_{2}$ be the centers of $Z_{1}$ and $Z_{2}$, respectively, and $Z_{2}'$ be the result of translating $Z_{2}$ along the direction of $\overrightarrow{p_{1}p_{2}}$ so that $Z_{1}$ and $Z_{2}'$ intersect only at their boundaries; $p_{2}'$ will denote the center of $Z_{2}'$ (see \ref{fig:10} a). Now, let $t_{1}$ and $t_{2}$ be the intersection points of the segment $p_{1}p_{2}'$ with the boundaries of $Z_{1}$ and $Z_{2}'$, respectively. Consider a facet $f_{1}$ of $Z_{1}$ which contains $t_{1}$, since $Z$ is centrally symmetric, there is a negative homothety from $Z_{1}$ to $Z_{2}'$, and this homothety maps $f_{1}$ into a facet $f_{2}$ of $Z_{2}'$ which contains $t_{2}$ and is parallel to $f_{1}$. Let $h_{1}$ and $h_{2}$ be the parallel hyperplanes that support $f_{1}$ and $f_{2}$, respectively, then $Z_{1}$ is contained in the halfspace determined by $h_{1}$ that contains $p_{1}$, while $Z_{2}$ is contained in the halfspace determined by $h_{2}$ that contains $p_{2}$. Suppose that $t_{1}\neq t_{2}$, then $p_{1}, t_{1}, t_{2}, p_{2}$ must lie on the segment $p_{1}p_{2}$ in that order and, by our previous observation, $Z_{1}$ and $Z_{2}$ would not intersect (see \ref{fig:10} 2b), it follows that $t_{1}=t_{2}$ and, thus, $f_{1}\cap f_{2}\neq\emptyset$. Now, since $f_{1}$ and $f_{2}$ are homothetic $d-1$ dimensional zonotopes and $f_{2}$ is not larger than $f_{1}$, the induction hypothesis implies the existence of a vertex $v$ of $f_{2}$ contained in $f_{1}$.

Let $w$ be the vertex of $Z_{2}$ which is mapped to $v$ by the translation from $Z_{2}$ to $Z_{2}'$, we claim that $w$ is contained in $Z_{1}$. The positive homothety from $Z_{2}$ to $Z_{1}$ maps $w$ to a vertex $w'$ of $Z_{1}$. The points $p_{1}, p_{2}, v, w$ and $w'$ all lie on the same plane and, since $Z_{2}$ is not larger than $Z_{1}$, $w'$ is contained in the closed region determined by the lines $wp_{1}$ and $wv$ which is opposite to $p_{2}$. This way, $w$ belongs to the convex hull of the points $p_{1}$, $v$ and $w'$; since these three points belong to the convex set $Z_{1}$, so does $w$ (see \ref{fig:10} c). This concludes the proof.
\end{proof}

\begin{figure}[!htbp]
\centering
\includegraphics[scale=0.62]{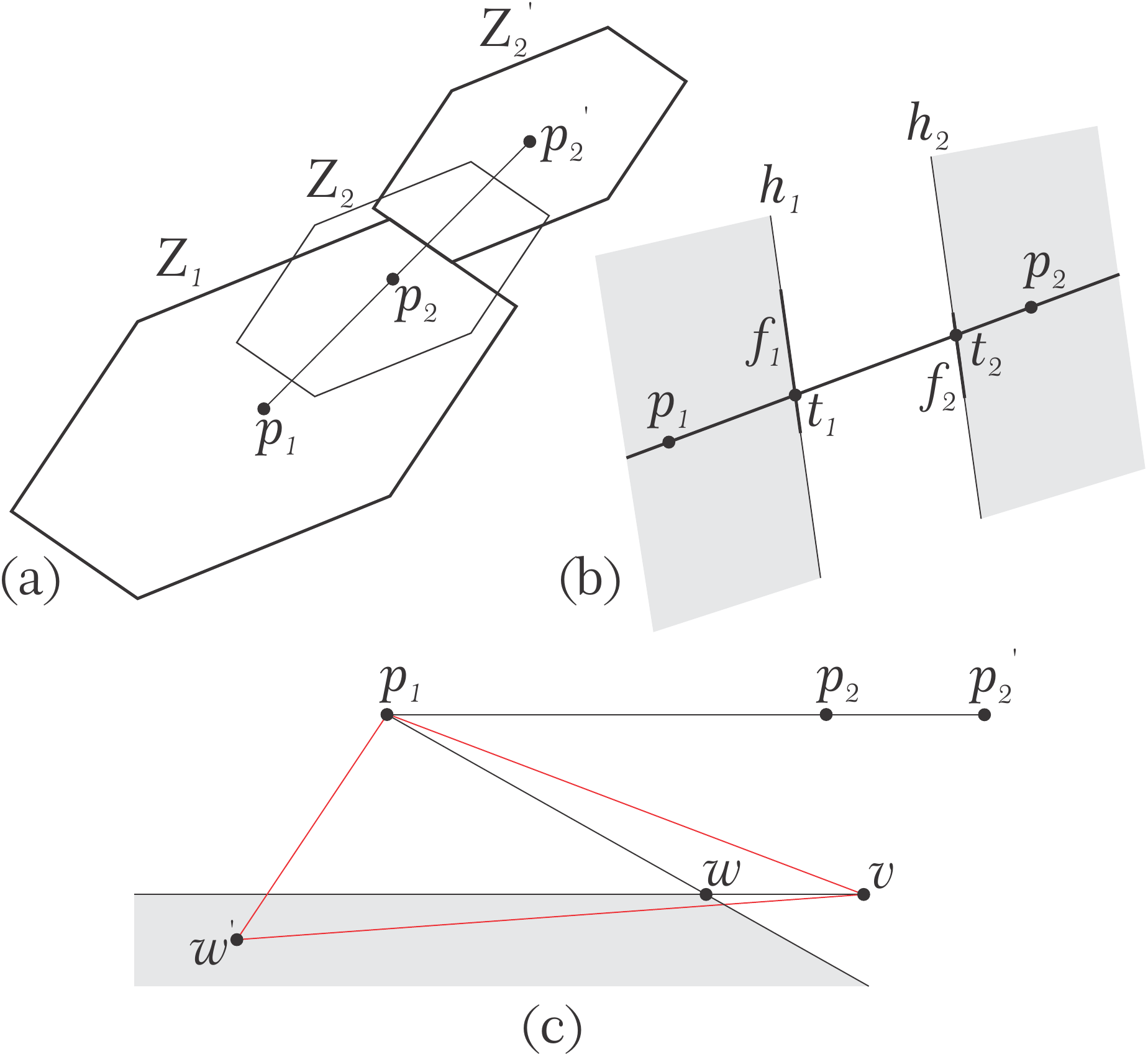}
\caption{(a): $Z_{1}$, $Z_{2}$ and $Z_{2}'$ (b): How the configuration would look if $t_{1}\neq t_{2}$ (c): Region where $w'$ lies highlighted in grey and triangle $w'vp_{1}$ in red.}
\label{fig:10}
\end{figure}

Proceeding as in the proof of Theorem \ref{teo:nets}, we get the following corollary, which generalizes a result for hypercubes by Kulkarni and Govindarajan \cite{weaknets}.

\begin{corollary}\label{teo:zonotopenets}
Let $Z\subset\mathbb{R}^d$ be a zonotope with $v$ vertices and denote by $\mathcal{H}_{Z}$ the family of all homothets of $C$. Then, for any finite set $S\subset\mathbb{R}^d$ and any $\epsilon>0$, $(S, \mathcal{H}_{Z}|_S)$ admits a weak $\epsilon$-net of size $\frac{v}{\epsilon}$.
\end{corollary}

We also have the following variant of Lemma \ref{teo:neighborhoodcover}.

\begin{lemma}\label{teo:zonotopeneighborhood}
Let $Z\subset\mathbb{R}^d$ be a zonotope and denote by $I$ the number of pairs $(f,v)$ where $f$ is a facet of $Z$ and $v$ is a vertex of $f$. Let $P\subset\mathbb{R}^d$ be a finite set and consider a collection of homothets $\lbrace Z_p\rbrace_{p\in P}$ of $Z$ such that $Z_p$ is of the form $p+\lambda Z$ and $\bigcap_{p\in P} Z_p\neq\emptyset$. Then there is a subset $P'$ of $P$ of size at most $I$ such that $\lbrace Z_p\rbrace_{p\in P'}$ covers $P$.
\end{lemma}

\begin{proof}
    Assume that and $O\in\bigcap_{p\in P} Z_p$ and that $O$ is the center of $Z$. Let $(f,v)$ be a pair as in the statement of the lemma and consider the homothet $Z'$ that results from applying a dilation to $Z$ with center $v$ and ratio $\frac{1}{2}$, the intersection of $f$ with this homothet will be denoted by $f_v$. Repeating this for every pair $(f,v)$, we obtain a decomposition of the facets of $Z$ into $I$ interior disjoint regions. 

    Now, for every pair $(f,v)$, let $P_{f,v}$ consist of all the points $p\in P$ with the property that the ray $\overrightarrow{Op}$ has non-empty intersection with $f_v$. Note that each element of $P$ belongs to at least one the aforementioned sets. From every $P_{f,v}$, choose an element which is maximal with respect to the norm with unit ball $Z$ and add it to $P'$; it is not hard to see that any homothet of $Z$ that is centered at this point and contains $O$ must cover every point in $P_{f,v}$. This way, $P'\leq I$ and $\lbrace Z_p\rbrace_{p\in P'}$ covers the union of all sets of the form $P_{f,v}$, which is $P$.
\end{proof}

Plugging the bounds given by Corollary \ref{teo:zonotopenets} and Lemma \ref{teo:zonotopeneighborhood} into the proof of Theorem \ref{teo:fbound} we obtain the following:  If $Z\subset\mathbb{R}^d$ is a zonotope with $V$ vertices and $I$ is as in the statement of lemma \ref{teo:zonotopeneighborhood} then, for any positive integer $k$ and any non-$\frac{k}{2}/C$-degenerate finite set of points $S\subset\mathbb{R}^d$, $f(Z,k,S)=\frac{2VI|S|}{k}$.

\section*{Acknowledgements}
I am grateful to Jorge Urrutia for many helpful discussions and, particularly, for suggesting that the fatness of the convex body might play an important role in the proofs of theorems \ref{teo:fbound} and \ref{teo:gbound}.

\bibliographystyle{plain}
\bibliography{refs}

\end{document}